\documentclass{amsart}
\usepackage{amsmath}
\usepackage{color}
\usepackage{mathtools}

\newtheorem{theorem}{Theorem}[section]
\newtheorem{lemma}[theorem]{Lemma}
\newtheorem{proposition}[theorem]{Proposition}
\newtheorem{corollary}[theorem]{Corollary}

 \theoremstyle{definition}
\newtheorem{definition}[theorem]{Definition}
\newtheorem{example}[theorem]{Example}

\theoremstyle{remark}

\numberwithin{equation}{section}

\begin{document}

\title[Ricci forms on noncompact complex manifolds]
{On Ricci forms of canonical metrics over noncompact complex manifolds}

\author{Hanzhang Yin}
\address{School of Mathematics, Harbin Institute of Technology,
         Harbin, Heilongjiang 150001, China}
\email{YinHZ@hit.edu.cn}

\begin{abstract}

In this paper, we study several types of geometric problems related to the Ricci curvature on noncompact complex manifolds, such as the existence of K\"{a}hler-Einstein metrics on complete K\"{a}hler manifolds with negative Ricci curvature, which can be seen as an improvement of the main theorem in Cheng-Yau \cite{CY80}; the existence of canonical Hermitian metrics with prescribed Ricci curvature on complete Hermitian manifolds, which can be regarded as noncompact versions of the Gauduchon conjecture on certain complete complex surfaces. Our method can also be used to construct Hesse-Einstein metrics in affine differential geometry.

\noindent{Keywords: K\"{a}hler-Einstein metrics; Gauduchon conjecture; Noncompact complex manifolds; Monge-Amp\`{e}re equations.}

\end{abstract}

\maketitle

\section{Introduction}

Let $M$ be a complete Hermitian manifold of dimension $n$. In this paper we consider the following prescribed Ricci curvature problem:
\begin{equation}\label{e1.7}
{\rm Ric}(\omega)=\Psi,
\end{equation}
where ${\rm Ric}(\omega)$ is the Chern-Ricci form of some complete canonical Hermitian metric $\omega=\sqrt{-1}{\omega_{i\bar{j}}}dz^i\wedge d\bar{z}^j$, and $\Psi$ is a given real $(1,1)$-form. For a Hermitian metric $g$ with K\"{a}hler form $\omega$, ${\rm Ric}(\omega)=-\sqrt{-1}\partial\overline{\partial}\log (\det g)$.

If $(M,\alpha)$ is a compact K\"{a}hler manifold (i.e. $d\alpha=0$), Yau's celebrated solution \cite{Yau78} of Calabi conjecture says that given any smooth representative $\Psi$ of the first Chern class $c_1(M)$, there exists a unique K\"{a}hler metric $\omega$ cohomologous to $\alpha$ such that \eqref{e1.7} holds.

Suppose that $\Psi=\lambda \omega$, where $\lambda$ is a constant, the K\"{a}hler metric $\omega$ that satisfies \eqref{e1.7} is called the K\"{a}hler-Einstein metric. The existence of the K\"{a}hler-Einstein metrics on compact K\"{a}hler manifolds has been established in the following cases: When $c_1(M)<0$, there exists a K\"{a}hler metric $\omega$ such that ${\rm Ric}(\omega)=-\omega$, see Aubin \cite{Aubin76} and Yau \cite{Yau78}; When $c_1(M)=0$, there exists a K\"{a}hler metric $\omega$ such that ${\rm Ric}(\omega)=0$, see Yau \cite{Yau78}; When $M$ is a Fano manifold (i.e. a projective manifold with $c_1(M)>0$), there exists a K\"{a}hler metric $\omega$ such that ${\rm Ric}(\omega)=\omega$ if and only if $M$ is K-stable, see Chen-Donaldson-Sun \cite{CDS15a,CDS15b,CDS15c} and Tian \cite{Tian97,Tian15}.

The existence of K\"{a}hler metrics with negative scalar curvature can be reduced to solving the following type of Monge-Amp\`{e}re equations
\begin{equation}\label{e1.8}
(\omega+\sqrt{-1}\partial\overline{\partial} u)^n=e^{u+h}\omega^n,
\end{equation}
where $u$ is the unknown function to be solved for, and $h$ is a given function. In the case of a noncompact K\"{a}hler manifold $(M,\omega)$, Cheng-Yau \cite{CY80} proved that if a complex manifold admits a complete K\"{a}hler
metric whose Ricci curvature is bounded from above by a negative constant, whose injectivity radius is bounded from below and whose curvature tensor
and its covariant derivatives are bounded, then the manifold admits a unique complete K\"{a}hler-Einstein metric metric.

The first goal of this paper is to remove the restrictions on the injectivity radius and covariant derivatives of the curvature tensor mentioned in the above theorem. We prove the following theorem:
\begin{theorem}\label{t1.1}
Let $(M,\alpha)$ be a complete K\"{a}hler manifold such that its curvature tensor $|Rm|<\infty$ and ${\rm Ric}(\alpha)\leq-c\alpha$ for some constant $c>0$. Then $M$ admits a unique complete K\"{a}hler-Einstein metric with scalar curvature equal to $-1$.
\end{theorem}
Wu-Yau \cite{WY20} obtained complete K\"{a}hler-Einstein metrics on complete K\"{a}hler manifolds with holomorphic sectional curvature bounded from above and below by two negative constants. Our Theorem \ref{t1.1} assumes that $(M,\alpha)$ has negative Ricci curvature instead of negative holomorphic sectional curvature.

To prove Theorem \ref{t1.1}, we need to consider equation \eqref{e1.8} on Hermitian manifolds. In fact, we can generalize \eqref{e1.8} to a general class of fully non-linear equations. Fix a real $(1,1)$-form $\chi$ on a complete noncompact Hermitian manifold $(M,\omega)$, for any $C^2$ function $u:M\rightarrow \mathbb{R}$ we have a new real $(1,1)$-form $g=\chi+\sqrt{-1}\partial\bar{\partial}u$, and we can define the endomorphism of $T^{1,0}M$ given by $A_j^i=\omega^{i\bar{p}}g_{j\bar{p}}$. Consider the equation for $u$ as follows:
\begin{equation}\label{e1.1}
F(A)=h(x)+\epsilon u,
\end{equation}
for a given function $h$ on $M$, where $\epsilon>0$ is a constant and
\begin{equation}
F(A)=f(\lambda_1,\ldots,\lambda_n)
\end{equation}
is a smooth symmetric function of the eigenvalues of $A$. Such equations without $\epsilon u$ on the right-hand side have been studied by Sz\'{e}kelyhidi \cite{S18} on compact Hermitian manifolds. Guo-Phong \cite{GP24} established the $L^\infty$-estimates for general classes of fully non-linear equations on Hermitian manifolds, their method can also be applied to open manifolds with a positive lower bound on their injectivity radii.

There are several structure conditions for equation \eqref{e1.1}:

\textbf{(a1)} $f$ is defined on an open symmetric convex cone
$\Gamma \subseteq \mathbb{R}^n$ and $\Gamma\neq \mathbb{R}^n$, containing the positive orthant $\Gamma_n=\{(x_1,\ldots,x_n)\in \mathbb{R}^n:x_i>0,i=1,\ldots,n\}$.

\textbf{(a2)} $f$ is symmetric, smooth, concave and increasing, i.e. its partials satisfy $f_i>0$ for all $i$.

\textbf{(a3)} $\sup_{\partial \Gamma}f<\inf_M h$.

\textbf{(a4)} For all $\mu\in \Gamma$, we have
\begin{equation}
\lim_{t\rightarrow\infty}f(t\mu)=\sup_{\Gamma} f,
\end{equation}
where both sides are allowed to be $\infty$.

The following definition was introduced by Guan \cite{Guan14} and Sz\'{e}kelyhidi \cite{S18}, we restate it here for application in our case.

\begin{definition}\label{d4.1}
We say that $u$ is a $\mathcal{C}$-subsolution for the equation $F(A)=h$ (see \eqref{e1.1} without $\epsilon u$ on the right-hand side) if the following holds. We require that for every point $x\in M$, if $\lambda=(\lambda_1,\ldots,\lambda_n)$ denote the eigenvalues of the endomorphism $\alpha^{i\bar{p}}g_{j\bar{p}}$ at $x$, then for all $i=1,\ldots,n$ we have
\begin{equation}\label{ne4.5}
\lim_{t\rightarrow\infty}f(\lambda+t\textbf{e}_i)>h(x)
\end{equation}
when $h(x)$ is independent of $u$. In our case \eqref{e1.1}, we instead require
\begin{equation}
\lim_{t\rightarrow\infty}f(\lambda+t\textbf{{e}}_i)=\infty.
\end{equation}
Here $\textbf{e}_i$ denotes the $i$th standard basis vector. Note that part of the requirement is that $\lambda+t\textbf{e}_i\in \Gamma$ for sufficiently large $t$, for the limit to be defined.
\end{definition}

We obtain the following existence results for solutions of \eqref{e1.1}.
\begin{theorem}\label{t7.6}
Let $(M,\alpha)$ be a complete Hermitian manifold, $\chi$ be a real $(1,1)$-form which is uniformly equivalent to $\alpha$, and $h$ be a smooth function.
Assume that
\begin{equation}
|T|_{\alpha}+|\nabla_{\alpha} T|_{\alpha}+|Rm|_{\alpha}+\sum_{k=0}^2|\nabla_{\alpha}^k \chi|_{\alpha}+\sum_{k=0}^2|\nabla_{\alpha}^k h|_{\alpha}<\infty.
\end{equation}
Suppose that the equation \eqref{e1.1} satisfies conditions \emph{\textbf{(a1)}}-\emph{\textbf{(a4)}}, $\Gamma= \Gamma_n$ and $F(\alpha^{i\bar{p}}\chi_{j\bar{p}})=0$, suppose $0$ is a $\mathcal{C}$-subsolution of \eqref{e1.1}. Then the equation \eqref{e1.1} exists a unique solution $u\in C^\infty(M)$ such that $\chi+\sqrt{-1}\partial\bar{\partial}u$ is uniformly equivalent to $\alpha$.
\end{theorem}

\textbf{Remark}: When we take $f(\lambda_1,\ldots,\lambda_n)=\log\lambda_1\ldots\lambda_n$,  $\Gamma=\Gamma_n$, \eqref{e1.1} is the complex Monge-Amp\`{e}re equation, hence we can use Theorem \ref{t7.6} to solve equations like \eqref{e1.8}. A related equation, the $(n-1)$ Monge-Amp\`{e}re equation was introduced by Fu-Wang-Wu \cite{FWW10}, in this case we take $f=\log \widetilde{\lambda_1}\ldots\widetilde{\lambda_n}$, $\Gamma=\Gamma_n$, where $\widetilde{\lambda_k}=\frac{1}{n-1}\sum_{i\neq k}\lambda_i$
for each $k$. Note that a class of $(n-1)$ Monge-Amp\`{e}re equations satisfy conditions \textbf{(a1)}-\textbf{(a4)}, see \cite{STW17}, further studies on such equations can be found in \cite{FWW15,JL23,TW17,TW19}.

Another geometric problem related to equation \eqref{e1.7} is the Gauduchon conjecture: Let $M$ be a compact complex manifold of dimension $n$ and $\Psi$ be a closed real $(1,1)$-form on $M$ with $[\Psi]=c_1^{{\rm BC}}(M)\in H_{{\rm BC}}^{1,1}(M,\mathbb{R})$. Then there is a Gauduchon metric (i.e. $\partial\overline{\partial}(\omega^{n-1})=0$) on $M$ satisfies \eqref{e1.7}. Where
\[H_{{\rm BC}}^{1,1}(M,\mathbb{R})=\frac{\{d-{\rm closed\;real }\;(1,1)-{\rm forms}\}}{\{\sqrt{-1}\partial\overline{\partial}\psi:\psi\in C^\infty(M,\mathbb{R})\}}\]
and $c_1^{{\rm BC}}(M)=[{\rm Ric}(\omega)]\in H_{{\rm BC}}^{1,1}(M,\mathbb{R})$. Sz\'{e}kelyhidi-Tosatti-Weinkove \cite{STW17} solved the above conjecture by studying $(n-1)$ Monge-Amp\`{e}re equations. Motivated by this work, we study the following prescribed Ricci curvature problem on certain noncompact Hermitian manifolds $(M^n,\alpha)$ with
\[\partial\overline{\partial}\alpha=0,\;\;\;\; \partial\overline{\partial}\alpha^2=0.\]
Such metrics were studied by Fino-Tomassini \cite{FT11}, they proved that $\partial\overline{\partial}\alpha=0$, $\partial\overline{\partial}\alpha^2=0$ is equivalent to that $\partial\overline{\partial}\alpha^k=0$ for all $k=1,2,\ldots,n$. By using Theorem \ref{t7.6} and the $L^\infty$ estimates of certain complex Monge-Amp\`{e}re equations, we obtain the following:
\begin{theorem}\label{t1.4}
Let $(M^n,\alpha)$ be a complete noncompact Hermitian manifold with $\partial\overline{\partial}\alpha=0$, $\partial\overline{\partial}\alpha^2=0$ and $|T|_{\alpha}+|\nabla_{\alpha} T|_{\alpha}+|Rm|_{\alpha}<\infty$. Let $f$ be a smooth function on $M$. Assume the following:

\emph{(a)} the following Sobolev inequality is true
\begin{equation}
\Big(\int_M|\phi|^{\frac{2n}{n-1}}dV\Big)^{\frac{n-1}{n}}\leq C_1\int_M |\nabla\phi|^2dV
\end{equation}
for some $C_1>0$ and all $\phi\in C_0^\infty(M)$.

\emph{(b)}
\begin{equation}
|f|(x)\leq \frac{C_2}{1+\rho_0^{2+\epsilon}(x)};\;\;\;\sum_{k=1}^2|\nabla_{\alpha}^k f|_{\alpha}<\infty
\end{equation}
for some $C_2,\epsilon>0$, and all $x\in M$. Where $\rho_0$ is the distance function from a fixed $o\in M$.

\emph{(c)} There exists a constant $C_3>0$ such that
\begin{equation}
Vol(B_r(x))\leq C_3 r^{2n}
\end{equation}
for some $C_3>0$ and all $r$ where $Vol(B_r(x))$ is the volume of the geodesic ball with radius $r$ centered at some $x\in M$.

Let $\Psi={\rm Ric}(\alpha)-\sqrt{-1}\partial\overline{\partial}f$, then there is a complete Hermitian metric $\omega$ on $M$ satisfies $\partial\overline{\partial}\omega=0$, $\partial\overline{\partial}\omega^2=0$ and \eqref{e1.7}.
\end{theorem}

As the noncompact versions of Calabi conjecture on open manifolds, Tian-Yau \cite{TY90,TY91} studied the solvability of complex Monge-Amp\`{e}re equations in order to construct Ricci-flat K\"{a}hler metrics on quasiprojective manifolds. The parabolic versions of such equations are discussed in Chau-Tam \cite{CT10,CT11}. Theorem \ref{t1.4} can be regarded as a generalization of Proposition 4.1 in \cite{TY91} to Hermitian manifolds $(M^n,\alpha)$ with $\partial\overline{\partial}\alpha=0$, $\partial\overline{\partial}\alpha^2=0$. Guan-Li \cite{GL10} solved complex Monge-Amp\`{e}re equations on compact Hermitian manifolds $(M^n,\alpha)$ with $\partial\overline{\partial}\alpha=0$, $\partial\overline{\partial}\alpha^2=0$, our results can also be regarded as a noncompact version of Theorem 1.3 in \cite{GL10}. We can also generalize Theorem 1.1 of Tian-Yau \cite{TY90} to such Hermitian manifolds, see the following theorem:

\begin{theorem}\label{t1.5}
Let $(M^n,\alpha)$ be a complete noncompact Hermitian manifold with $\partial\overline{\partial}\alpha=0$, $\partial\overline{\partial}\alpha^2=0$ and $|T|_{\alpha}+|\nabla_{\alpha} T|_{\alpha}+|Rm|_{\alpha}<\infty$. Assume that $Vol(B_r(x_0))\leq C_2 r^{2}$ for all $r>0$, and $Vol(B_1(x))\geq C_2^{-1}(1+\rho(x_0,x))^{-\beta}$ for some fixed point $x_0$ in $M$ and some positive constants $C_2,\beta$ independent of $x$, where $\rho$ is the distance function on $M$, $B_r(x_0)$ is the geodesic ball with radius $r$ centered at $x_0\in M$. Let $f$ be a smooth function on $M$ satisfying:
\begin{equation}
\int_M(e^f-1)\omega^n=0
\end{equation}
\begin{equation}
|f|(x)\leq \frac{C_2}{(1+\rho(x_0,x))^{4+2\beta}};\;\;\;\sum_{k=1}^2|\nabla_{\alpha}^k f|_{\alpha}<\infty,
\end{equation}
for some constant $C_2$. Let $\Psi={\rm Ric}(\alpha)-\sqrt{-1}\partial\overline{\partial}f$, then there is a complete Hermitian metric $\omega$ on $M$ satisfies $\partial\overline{\partial}\omega=0$, $\partial\overline{\partial}\omega^2=0$ and \eqref{e1.7}.
\end{theorem}

Let $\dim M=n$, when $\partial\overline{\partial}\alpha=0$, the metric $\alpha$ is called \emph{strong KT} (see e.g. \cite{GHR84}); when $\partial\overline{\partial}\alpha^{n-2}=0$, in the terminology of Jost-Yau \cite{JY93} the metric $\alpha$ is called \emph{astheno-K\"{a}hler}; when $n=2$, the conditions $\partial\overline{\partial}\alpha=0$ and $\partial\overline{\partial}\alpha^2=0$ are equivalent to $\alpha$ being a Gauduchon metric. Therefore, Theorem \ref{t1.4} and Theorem \ref{t1.5} can be regarded as noncompact versions of the Gauduchon conjecture on certain complete complex surfaces. Some examples of manifolds that satisfy the conditions in Theorem \ref{t1.4} or Theorem \ref{t1.5} can be found in Section 5, see Example \ref{ex5.3}, \ref{ex5.4} and \ref{ex5.6}.

A related problem, the existence of Hesse-Einstein metrics on compact Hessian manifolds (affine K\"{a}hler manifolds) with negative first Chern class, was solved by Cheng-Yau \cite{CY82}, which is also called Einstein K\"{a}hler affine metric in \cite{CY82}. For the relationship between Hessian manifolds and K\"{a}hler manifolds, and the relationship between Einstein metrics on them, see Section 6. Puechmorel-T{\^o} \cite{PT23} used geometric flow to provide another method for constructing the above metrics. For more studies on affine differential geometry, see (\cite{CY77, HW17, LYZ05}).
In the following, we use the a priori estimates for \eqref{e1.1} to solve the real Monge-Amp\`{e}re equations and construct the Hesse-Einstein metrics on certain noncompact affine manifolds.
\begin{theorem}\label{t1.6}
Let $(M,D,g)$ be a complete Hessian manifold. Suppose that
\begin{equation}
\sum_{k=0}^4|D^k g|_g<\infty\;\;{\rm and}\;\;\kappa(g)\geq cg
\end{equation}
for some constant $c>0$. Then $M$ admits a unique complete Hesse-Einstein metric $g$ with $\beta(g)=-g$.
\end{theorem}
The rest of this paper is organized as follows. In Section 2, we provide some preliminaries which may be used later.
In Section 3, we establish the a priori estimates for \eqref{e1.1}. In Section 4, we prove Theorem \ref{t1.1} and \ref{t7.6}. In Section 5, we prove Theorem \ref{t1.4} and \ref{t1.5}. In Section 6, we prove Theorem \ref{t1.6}. We write an outline at the beginning of each section to introduce some key ideas.

\textbf{Acknowledgement.} I would like to thank Professor Jiao Heming and Wang Zhizhang suggesting him to read \cite{CY80} and to study the Monge-Amp\`{e}re type equations on noncompact Hermitian manifolds.

\section{Preliminaries}

In this section, we provide some basic results and preliminaries which may be used in the following sections. Some notations and conventions used in this paper will also be introduced at various points throughout this section.

First, we introduce the Chern connection, most of the materials can be found in \cite{LT20}.

$\bullet$
\begin{minipage}[t]{0.95\linewidth}
Let $(M,g)$ be a Hermitian manifold. Then \emph{Chern connection} of $g$ is defined as follows:  In local holomorphic
coordinates $z^i$, for a $T^{1,0}$ vector field $X=X^i \partial_i$ and a $T^{0,1}$ vector field $Y=Y^{\bar{i}} \partial_{\bar{i}}$, where $\partial_i:=\frac{\partial}{\partial z^i}$, $\partial_{\bar{i}}:=\frac{\partial}{\partial \bar{z}^i}$,
\[\nabla_k X^i=\partial_k X^i+\Gamma_{kj}^i X^j,\;\;\;\nabla_{\bar{k}}X^i=\partial_{\bar{k}}X^i\]
\[\nabla_k Y^{\bar{i}}=\partial_k Y^{\bar{i}},\;\;\;\nabla_{\bar{k}}Y^{\bar{i}}=\partial_{\bar{k}}Y^{\bar{i}}+\overline{\Gamma_{kj}^i}Y^{\bar{j}}\]
For a $(1,0)$ form $a=a_i {\rm d}z^i$ and a $(0,1)$ form $b_{\bar{i}}{\rm d}\bar{z^i}$,
\[\nabla_k a_i=\partial_k a_i-\Gamma_{ki}^j a_j,\;\;\;\nabla_{\bar{k}}a_i=\partial_{\bar{k}}a_i\]
\[\nabla_k b_{\bar{i}}=\partial_k b_{\bar{i}},\;\;\;\nabla_{\bar{k}}b_{\bar{i}}=\partial_{\bar{k}}b_{\bar{i}}-\overline{\Gamma^j_{ki}}b_{\bar{j}}\]
Here $\nabla_i:=\nabla_{\partial_i}$, and $\Gamma_{jk}^i$ are Christoffel symbols, with
\[\Gamma_{jk}^i=g^{i\bar{l}}\partial_j g_{k\bar{l}},\]
where $g^{i\bar{l}}g_{j\bar{l}}=\delta_{ij}$,
\[\delta_{ij}=\left\{ \begin{aligned}
&0, &\;\;\;i\neq j\\
&1,&\;\;\;i=j\\
\end{aligned} \right.\]
\end{minipage}

$\bullet$ The tensor operations in this paper follow the \emph{Einstein summation convention}.

$\bullet$
\begin{minipage}[t]{0.95\linewidth}
We extend covariant derivatives to act naturally on any type of tensor. For example, if $W$ is a tensor with components $W_{k}^{i\bar{j}}$, then define
\[\nabla_m W_{k}^{i\bar{j}}=\partial_m W_{k}^{i\bar{j}}+\Gamma_{ml}^i W_{k}^{l\bar{j}}-\Gamma_{mk}^l W_{l}^{i\bar{j}}\]
\[\nabla_{\bar{m}}W_{k}^{i\bar{j}}=\partial_{\bar{m}}W_{k}^{i\bar{j}}+\overline{\Gamma_{ml}^j}W_{k}^{i\bar{l}}\]
\end{minipage}

$\bullet$
\begin{minipage}[t]{0.95\linewidth}
The Hermitian metric $g$ defines a pointwise norm $|\cdot|_g$ on any tensor. For example, with $X,Y,a,b$ as above, we define
\[|X|_g^2=g_{i\bar{j}}X^i\overline{X^j},\;\;|Y|_g^2=g_{i\bar{j}}Y^{\bar{j}}\overline{Y^{\bar{i}}},\;\;|a|_g^2=g^{i\bar{j}}a_i\overline{a_j},\;\;|b|_g^2=g^{i\bar{j}}b_{\bar{j}}\overline{b_{\bar{i}}}.\]
This is extended to any type of tensor. For example, with $W$ as above, we define
\[|W|^2_g=g^{k\bar{l}}g_{i\bar{j}}g_{p\bar{q}}W_k^{i\bar{q}}\overline{W_l^{j\bar{p}}}.\]
\end{minipage}

$\bullet$
\begin{minipage}[t]{0.95\linewidth}
Noted that the Chern connection is a connection such that $\nabla g=\nabla J=0$ and the torsion has no $(1,1)$ component.  The torsion of $g$ is defined to be
\[T_{ij}^k=\Gamma_{ij}^k-\Gamma_{ji}^k.\]
We remark that $g$ is K\"{a}hler if and only if $T=0$. The associated \emph{Chern curvature tensor} for the metric $g$ is then defined as
\[R_{i\bar{j}k}^{\;\;\;\;\;l}=-\partial_{\bar{j}}\Gamma_{ik}^l.\]
In a local holomorphic coordinate system $\{z^1,\cdots,z^n\}$, we have
\[R_{i\bar{j}k\bar{l}}=-\frac{\partial^2g_{i\bar{j}}}{\partial z^k \partial \bar{z}^l}+g^{p\bar{q}}\frac{\partial g_{i\bar{q}}}{\partial z^k}\frac{\partial g_{p\bar{j}}}{\partial \bar{z}^l},\]
in this paper, we denote the curvature tensor $R_{i\bar{j}k\bar{l}}$ by $Rm$. The Chern-Ricci curvature is defined by
\[R_{i\bar{j}}=g^{k\bar{l}}R_{i\bar{j}k\bar{l}}=-\partial_i\partial_{\bar{j}}\log\det g.\]
Note that if $g$ is not K\"{a}hler, then $R_{i\bar{j}}$ may not equal to $g^{k\bar{l}}R_{k\bar{l}i\bar{j}}$.
\end{minipage}

The next lemma gives some relations between the Riemannian connection and the Chern connection.
\begin{lemma}\label{l2.1}
Suppose $|T|_g$ and $|\nabla T|_g$ are bounded. Then the Riemannian curvature is bounded if and only if the Chern curvature is bounded.
\end{lemma}

$\bullet$
\begin{minipage}[t]{0.95\linewidth}
For two real $(1,1)$-forms $\chi$ and $\chi'$ on $(M,g)$, we say that $\chi$ is uniformly equivalent to $\chi'$ if and only if
\[c^{-1}\chi'<\chi<c\chi'\]
for some constant $c>0$.
\end{minipage}

Next, we will introduce some definitions and propositions about complete noncompact complex manifolds. To facilitate subsequent discussions, we introduce the definition of bounded geometry, which was introduced by Cheng and Yau (see, for example, \cite{CY80}), originally to adapt the Schauder-type estimates
to solve PDEs on complete noncompact manifolds.
\begin{definition}\label{db2.1}
Suppose that $(M^n,g)$ is a complete Hermitian manifold. Let $k\ge 1$ be an integer and $0<\alpha<1$. $(M^n,g)$ is said to have\emph{ bounded geometry} of order $k+\alpha$ if there are positive numbers $r,\kappa_1,\kappa_2$ such that at every $p\in M$ there is a neighborhood $U_p$ of $p$, and  biholomorphism $\xi_p$ from $D(r):=\{z\in \mathbb{C}^n||z|<r\}$ onto $U_p$ with $\xi_p(0)=p$ satisfying the following properties:\\
\hspace*{0.5cm}(i) the pull back metric $\xi^{*}_p(g)$ satisfies:
\[\kappa_1 g_{\mathbb{C}^n}\le \xi^{*}_p(g)\le \kappa_2 g_{\mathbb{C}^n}\]
\hspace*{1cm}where $g_{\mathbb{C}^n}$ is the standard metric on $\mathbb{C}^n$;\\
\hspace*{0.4cm}(ii) the components $g_{i\bar{j}}$ of $\xi^{*}_p(g)$ in the natural coordinate of $D(r)\subset \mathbb{C}^n$ are
\hspace*{1cm}uniformly bounded in the standard $C^{k+\alpha}$ norm in $D(r)$ independent of $p$.\\
$g$ is said to have bounded geometry of infinity order if instead of $(ii)$ we have for any $k$, the $k$-$th$ derivatives of $g_{i\bar{j}}$ in $D(r)$ are bounded by a constant independent of $p$. We denote the covering holomorphic coordinate charts on $M$ with respect to $\{\xi_p(D(r));p\in M\}$ by $(V,(v^1,\cdots,v^n))$. $(M^n,g)$ is of \emph{quasi-bounded geometry} if $\xi_p$ is a nonsingular holomorphic map instead of a biholomorphism.

Fix a real $(1,1)$-form $\chi$ on $M$, we say that $\chi$ has bounded geometry of order $l$ with respect to $(M,g_{i\bar{j}})$ if and only if: In the coordinate chart $(V,(v^1,\cdots,v^n))$, the components $\chi_{i\bar{j}}$ of $\chi$ are uniformly bounded in the standard $C^{l+\alpha}$ norm in $(V,(v^1,\cdots,v^n))$ independent of $V$.
\end{definition}

Suppose $(M,g_{i\bar{j}})$ is a complete Hermitian manifold having bounded geometry of order $l+\beta$. Let $\{(V,(v^1,\cdots,v^n))\}$ be a family of holomorphic coordinate charts covering $M$ and satisfying the conditions of Definition \ref{db2.1}. Then, for any $u\in C^\infty(M)$, non-negative integer $k,k\leq l$, and $\alpha \in(0,1)$, we define norms $|u|_{k}$ and $|u|_{k+\alpha}$ to be
\begin{equation}
|u|_{k}=\sup_V\biggl(\sup_{z\in V}\biggl(\sum_{|\alpha|+|\beta|\leq k}\biggl|\frac{\partial^{|\alpha|+|\beta|}}{\partial v^\alpha \partial \bar{v}^\beta}u(z)\biggl|\biggl)\biggl)
\end{equation}
\begin{equation}\label{eb2.18}
\begin{aligned}
|u|_{k+\alpha}=&\sup_V\biggl(\sup_{z\in V}\biggl(\sum_{|\alpha|+|\beta|\leq k}\biggl|\frac{\partial^{|\alpha|+|\beta|}}{\partial v^\alpha \partial \bar{v}^\beta}u(z)\biggl|\biggl)\\
&+\sup_{z,z'\in V}\biggl(\sum_{|\alpha|+|\beta|= k}|z-z'|^{-\alpha}\biggl|\frac{\partial^{|\alpha|+|\beta|}}{\partial v^\alpha \partial \bar{v}^\beta}u(z)-\frac{\partial^{|\alpha|+|\beta|}}{\partial v^\alpha \partial \bar{v}^\beta}u(z')\biggl|\biggl)\biggl).
\end{aligned}
\end{equation}
The completion of $\{u\in C^\infty(M):|u|_{k+\alpha}<\infty\}$ with respect to $|\cdot|_{k+\alpha}$ is then a Banach space and will be denoted by $\tilde{C}^{k+\alpha}(M)$. Define
\begin{equation}
\tilde{C}^\infty(M)=\bigcap_{k=0}^\infty \tilde{C}^{k+\alpha}(M).
\end{equation}

\textbf{Remark}: By the Arzel\`{a}-Ascoli Theorem, it is obvious that $\tilde{C}^{k+\alpha}(M)\subset \{u\in C^k(M):|u|_{k}<\infty\}$. Therefore
\begin{equation}
\tilde{C}^\infty(M)=\bigcap_{k=0}^\infty \{u\in C^\infty(M):|u|_{k}<\infty\}.
\end{equation}

On the manifolds with bounded geometry, we have the generalized maximum principle (see \cite{CY80}), which is our main tool to derive the a priori estimates. For completeness, we state this in the following proposition.
\begin{proposition}\label{pb4.2}\emph{(Cheng-Yau)}
Suppose $(M,g_{i\bar{j}})$ is a complete Hermitian manifold. Suppose that, for any $x\in M$, there is an open set $D^x$ containing $x$ and non-negative function $\varphi^x:\overline{D^x}\rightarrow \mathbb{R}$ such that {\rm (i)} $\overline{D^x}$ is compact, {\rm (ii)} $\varphi^x_{(x)}=1$ and $\varphi^x=0$ on $\partial D^x$, {\rm (iii)} $\varphi^x\leq c$, $|\nabla \varphi^x|\leq c$ and $(\varphi_{i\bar{j}}^x)\geq -c(g_{i\bar{j}})$, where $c$ is a positive constant independent of $x$. Suppose $f$ is a function on $M$ which is bounded from above. Then there exists a sequence $\{x_i\}$ in $M$ such that $\lim f(x_i)=\sup f$, $\lim|df(x_i)|=0$ and $\overline{\lim}(f_{p\bar{q}}(x_i))\leq 0$, where the Hessian is taken with respect to $(g_{i\bar{j}})$.
\end{proposition}
\begin{corollary}\label{pbb4.3}\emph{(Cheng-Yau)}
Suppose $(M,g_{i\bar{j}})$ is a complete Hermitian manifold with bounded geometry of order $l$, $l\geq 0$. Then the assertion of Proposition \ref{pb4.2} is valid on $(M,g_{i\bar{j}})$.
\end{corollary}

\begin{lemma}\label{l2.5}
Suppose $(M,g_{i\bar{j}})$ is a complete Hermitian manifold such that $|T|_g$, $|\nabla T|_g$ and $|Rm|_g$ are bounded. Then the assertion of Proposition \ref{pb4.2} is valid on $(M,g_{i\bar{j}})$.
\end{lemma}
\begin{proof}
By Lemma \ref{l2.1} and \ref{l7.2}, there exists a exhaustion function $\tilde{\rho}$ on $(M,g_{i\bar{j}})$ as in Lemma \ref{l7.2}. Then for any $x\in M$, define $\varphi^x:=-\tilde{\rho}+\tilde{\rho}(x)+1$ and $D^x:=\{y\in M|\varphi^x(y)>0\}$, where $\varphi^x$ and $D^x$ are as in Proposition \ref{pb4.2}.
\end{proof}

\section{The a priori estimates for \eqref{e1.1}}

In this section, we consider the a priori estimates of \eqref{e1.1} on Hermitian manifolds with bounded geometry, most of the calculations in this section can be found in my preprint \cite{Yin25}, but for the convenience of the reader, we provide a complete proof. The main difficulty lies in the $C^{2}$ estimates. Since a bounded function on a noncompact manifold may not attain its extremum at some point, the key idea of this section is to apply the generalized maximum principle (see \eqref{eb4.37}, \eqref{eb4.38} and \eqref{eb4.39}), rather than the classical maximum principle, to establish the $C^{2}$ estimates. We use some of the language and approaches of Sz\'{e}kelyhidi-Tosatti-Weinkove \cite{STW17}.

Let $(M,\alpha)$ be a Hermitian manifold with bounded geometry of order 2 of complex dimension $n$ and write
\begin{equation}
\alpha=\sqrt{-1}\alpha_{i\bar{j}}dz^i\wedge d\bar{z}^j>0.
\end{equation}
Fix a background $(1,1)$-form $\chi=\sqrt{-1}\chi_{i\bar{j}}dz^i\wedge d\bar{z}^j$ which is not necessarily positive definite. For $u:M\rightarrow \mathbb{R}$, define a new tensor $g_{i\bar{j}}$ by
\begin{equation}\label{e4.3}
g_{i\bar{j}}:=\chi_{i\bar{j}}+u_{i\bar{j}}.
\end{equation}
And we can define the endomorphism of $T^{1,0}M$ given by $A_j^i=\alpha^{i\bar{p}}g_{j\bar{p}}$.

In the following, we consider the a priori estimates of the equations for $u\in C^4(M)\cap \{u\in C^2(M)|\;|u|_2<\infty\}$ as follows:
\begin{equation}\label{en4.3}
F(A)=h(x)+\epsilon u
\end{equation}
for a given function $h\in \{h\in C^2(M)|\;|h|_2<\infty\}$ on $M$, where $\epsilon\in[0,1]$ is a constant, and
\begin{equation}
F(A)=f(\lambda_1,\ldots,\lambda_n)
\end{equation}
is a smooth symmetric function of the eigenvalues of $A$, and the equation \eqref{en4.3} satisfies conditions \textbf{(a1)}-\textbf{(a4)} and $\chi$ has bounded geometry of order 2 with respect to $(M,\alpha)$.

Note that, if $\underline{u}$ is a $\mathcal{C}$-subsolution and $\underline{u}\in \tilde{C}^{\infty}(M)$, then replacing $\chi$ by
\begin{equation}
\chi'_{i\bar{j}}=\chi_{i\bar{j}}+\underline{u}_{i\bar{j}},
\end{equation}
it is easy to see $\chi'$ also has bounded geometry of order 2. We may assume that $\underline{u}=0$. The important consequence of $0$ being a $\mathcal{C}$-subsolution is the Proposition 2.3 of \cite{STW17}.

Our main goal is the following estimate:
\begin{equation}\label{eb4.17}
\sup_M|\sqrt{-1}\partial\bar{\partial}u|_\alpha\leq C\biggl(\sup_M|\nabla u|_\alpha^2+1\biggl).
\end{equation}
In fact, the estimate \eqref{eb4.17} is equivalent to the bound
\begin{equation}
\lambda_1\leq CK,
\end{equation}
where $K=\sup_M|\nabla u|_\alpha^2+1$ and $\lambda_1$ is the largest eigenvalue of $A=(A_j^i)=(\alpha^{i\bar{p}}g_{j\bar{p}})$. Indeed, our assumption on the cone $\Gamma$ implies that $\sum_i \lambda_i>0$ (see \cite{STW17}). Then, if $\lambda_1$ is bounded from above by $CK$, so is $|\lambda_i|$ for all $i$, giving the same bound for $\sup_M|\sqrt{-1}\partial\bar{\partial}u|_\alpha$.

The generalized maximum principle on non-compact manifolds (see Proposition \ref{pb4.2}) are our main tools in this section.

To use the generalized maximum principle, in the sequel, we assume that $u\in C^4(M)\cap \{u\in C^2(M)|\;|u|_2<\infty\}$. We consider the function
\begin{equation}
H=\log \lambda_1+\phi(|\nabla u|_\alpha^2)+\psi(u),
\end{equation}
where $\phi$ is defined by
\begin{equation}
\phi(t)=-\frac{1}{2}\log \biggl(1-\frac{t}{2K}\biggl),
\end{equation}
so that $\phi(|\nabla u|_\alpha^2)\in [0,\frac{1}{2}\log 2]$ satisfies
\begin{equation}
\frac{1}{4K}<\phi'<\frac{1}{2K},\;\;\;\;\phi''=2(\phi')^2,
\end{equation}
and $\psi$ is defined by
\begin{equation}
\psi(t)=D_1e^{-D_2 t},
\end{equation}
for sufficiently large uniform constants $D_1,D_2>0$ to be chosen later. By the $C^0$ bound on $u$, the quantity $\psi(u)$ is uniformly bounded.


The quantity $H$ need not be smooth at any point, because the largest eigenvalue of $A$ may have eigenspace of dimension larger than 1. Therefore, we can not apply Proposition \ref{pb4.2} to $H$ directly. Since $u\in C^4(M)\cap \{u\in C^2(M)|\;|u|_2<\infty\}$, $H$ is bounded from above in $M$, there exists a a sequence $\{p_i\}$ in $M$ such that $\lim H(p_i)=\sup H$. If $H$ achieves its maximum at some point $p\in M$, then we can use the same method as in \cite{STW17} to obtain the estimate \eqref{eb4.17}. Hence in the following, we assume that
\begin{equation}\label{ne4.13}
H<\sup H \;\;{\rm on} \;\;M.
\end{equation}

For some $i$, we work at the point $p_i$, where $i$ is a sufficiently large positive integer to be chosen later. In orthonormal complex coordinates for $\alpha$ centered at this point, such that $g$ is diagonal and $\lambda_1=g_{1\bar{1}}$. To take care of this, we carry out a perturbation as in \cite{S18}, choosing local coordinates such that $H$ achieves $H(p_i)$ at the origin, where $A$ is diagonal with eigenvalues $\lambda_1\geq \ldots \geq\lambda_n$. We fix a diagonal matrix $B$ independent of $p_i$, with $B_1^1=0$ and $0<B_2^2<\ldots<B^n_n$, and we define $\tilde{A}=A-B$, denoting its eigenvalues by $\tilde{\lambda}_1,\ldots,\tilde{\lambda}_n$.

At the origin, we have
\begin{equation}
\tilde{\lambda}_1=\lambda_1\;\;\;\;{\rm and}\;\;\;\;\tilde{\lambda}_i=\lambda_i-B_i^i,\;\;i>1,
\end{equation}
and $\tilde{\lambda}_1>\tilde{\lambda}_2>\ldots>\tilde{\lambda}_n$. As discussed above, our assumption on the cone $\Gamma$ implies that $\sum_i \lambda_i>0$, and we fix the matrix $B$ small enough so that
\begin{equation}
\sum_i \tilde{\lambda}_i>-1.
\end{equation}
We can choose this matrix $B$ in such a way that, in addition,
\begin{equation}\label{eb4.25}
\sum_{p>1}\frac{1}{\lambda_1-\tilde{\lambda}_p}\leq C,
\end{equation}
for some fixed constant $C$ depending on the dimension $n$. Next, we will show that $\tilde{\lambda}_i$ is smooth on $B_r(0)$ for some $r>0$ independent of $p_i$. Define
\begin{equation}\label{eb4.26}
G(\lambda,\tilde{A}(x))=\det(\lambda I-\tilde{A}(x)).
\end{equation}
Obviously, by \eqref{eb4.25} and \eqref{eb4.26}, we have
\begin{equation}\label{eb4.27}
\begin{aligned}
\frac{\partial G(\lambda,\tilde{A}(x))}{\partial \lambda}\biggl|_{x=0,\lambda=\tilde{\lambda}_1(0)}&=\prod_{i=2}^{n}({\lambda}_1(0)-\tilde{\lambda}_i(0))\\
&\geq \biggl(\frac{n-1}{\sum_{p>1}\frac{1}{\lambda_1(0)-\tilde{\lambda}_p(0)}}\biggl)^{n-1}\\
&\geq \biggl(\frac{n-1}{C}\biggl)^{n-1}.
\end{aligned}
\end{equation}
By implicit function theorem, there exists a smooth function $\lambda=\lambda^1(\tilde{A})$ such that
\begin{equation}\label{eb4.28}
G(\lambda^1(\tilde{A}),\tilde{A})=0\;\;\;\;{\rm and }\;\;\;\;\lambda^1(\tilde{A}(0))=\lambda_1(0)
\end{equation}
on a neighborhood $B_{r'}(A(0))$ of $A(0)$($A$ can be seen as a vector on $\mathbb{C}^{n\times n}$). Since the matrix $A$ depends only on $\chi_{i\bar{j}}$, $u_{i\bar{j}}$ and $\alpha_{i\bar{j}}$, we can assume that $r'$ is a positive constant independent of $p_i$ and on $B_{r''}(0)\subseteq M$ we also have \eqref{eb4.28} for some $r''>0$ independent of $p_i$. Similarly, we obtain $\lambda=\lambda^i(\tilde{A})$ such that
\begin{equation}\label{eb4.29}
G(\lambda^i(\tilde{A}),\tilde{A})=0\;\;\;\;{\rm and }\;\;\;\;\lambda^i(\tilde{A}(0))=\tilde{\lambda}_i(0)
\end{equation}
on $B_{r''}(0)\subseteq M$. Differentiating the \eqref{eb4.29}, we get
\begin{equation}
dG(\lambda^i(\tilde{A}),\tilde{A})=\frac{\partial G}{\partial \lambda^i}d\lambda^i+\sum_{p,q=1}^n\frac{\partial G}{\partial A_p^q}dA_p^q=0.
\end{equation}
So $d\lambda^i$ can be small enough if $r''\rightarrow 0$, we have a constant $r>0$ independent of $p_i$ such that
\begin{equation}
\lambda^1(\tilde{A}(x))>\lambda^2(\tilde{A}(x))>\ldots>\lambda^n(\tilde{A}(x))
\end{equation}
on $B_r(0)$, so $\tilde{\lambda}_1=\lambda^1(\tilde{A}(x))$ is smooth on $B_r(0)$.

Now, after possibly shrinking the chart, the quantity
\begin{equation}
\widetilde{H}=\log \tilde{\lambda}_1+\phi(|\nabla u|_\alpha^2)+\psi(u)
\end{equation}
is smooth on the chart $B_r(p_i)$, in the sequel, we compute in this chart. Let $\sup H=L$, by \eqref{ne4.13}, we have $L-\widetilde{H}>0$ on $M$. Let $\{p_i\}$ be a sequence in $M$ as above such that $\lim H(p_i)=\sup H$. Since $(M,\alpha)$ has bounded geometry of order 2, then there is a non-negative function $\varphi^{p_i}:B_r(p_i)\rightarrow \mathbb{R}$ such that {\rm (i)} $\varphi^{p_i}_{(p_i)}=1$ and $\varphi^{p_i}=0$ on $\partial B_r(p_i)$, {\rm (ii)} $\varphi^{p_i}\leq c$, $|\nabla \varphi^{p_i}|\leq c$ and $(\varphi_{i\bar{j}}^{p_i})\geq -c(\alpha_{i\bar{j}})$, where $c$ is a positive constant independent of $i$. Consider $(L-\widetilde{H})/\varphi^{p_i}$ as a function defined on $B_r(p_i)$. Let $x_i\in B_r(p_i)$ be such that
\begin{equation}
\biggl(\frac{L-\widetilde{H}}{\varphi^{p_i}}\biggl)(x_i)=\inf_{B_r(p_i)}\biggl(\frac{L-\widetilde{H}}{\varphi^{p_i}}\biggl).
\end{equation}
Then,
\begin{equation}
\frac{L-\widetilde{H}}{\varphi^{p_i}}(x_i)\leq \frac{L-\widetilde{H}}{\varphi^{p_i}}(p_i)=L-H(p_i),
\end{equation}
\begin{equation}
\frac{d(L-\widetilde{H})}{L-\widetilde{H}}(x_i)=\frac{d\varphi^{p_i}}{\varphi^{p_i}}(x_i),
\end{equation}
\begin{equation}
\biggl(\frac{(L-\widetilde{H})_{p\bar{q}}}{L-\widetilde{H}}(x_i)\biggl)\geq \biggl(\frac{\varphi^{p_i}_{p\bar{q}}}{\varphi^{p_i}}(x_i)\biggl).
\end{equation}
Using the property of $\varphi^{p_i}$ we see that
\begin{equation}\label{eb4.37}
0<L-\widetilde{H}(x_i)\leq c(L-H(p_i)),
\end{equation}
\begin{equation}\label{eb4.38}
|d\widetilde{H}(x_i)|\leq c(L-H(p_i)),
\end{equation}
\begin{equation}\label{eb4.39}
(\widetilde{H}_{p\bar{q}}(x_i))\leq c(L-H(p_i))(\alpha_{p\bar{q}})
\end{equation}

We will apply \eqref{eb4.37}, \eqref{eb4.38} and \eqref{eb4.39} to $\widetilde{H}$ instead of the classical maximum principles. Our goal is to obtain the bounded $\tilde{\lambda}_1\leq CK$ at $x_i$, which will give us the required estimate \eqref{eb4.17}. Hence, we may and do assume that $\tilde{\lambda}_1\gg K$ at this point. We now differentiate $\widetilde{H}$ at $x_i$ and, we use subscripts $k$ and $\bar{l}$ to denote the partial derivatives $\partial/\partial z^k$ and $\partial/\partial \bar{z}^l$. We have
\begin{equation}\label{eb4.40}
\begin{aligned}
\widetilde{H}_k&=\frac{\tilde{\lambda}_{1,k}}{\tilde{\lambda}_1}+\phi'(\alpha^{p\bar{q}}u_p u_{\bar{q}k}+\alpha^{p\bar{q}}u_{pk}u_{\bar{q}}+(\alpha^{p\bar{q}})_k u_p u_{\bar{q}})+\psi'u_k\\
&=\frac{\tilde{\lambda}_{1,k}}{\tilde{\lambda}_1}+\phi'(u_p u_{\bar{p}k}+u_{pk}u_{\bar{p}}+(\alpha^{p\bar{q}})_k u_p u_{\bar{q}})+\psi'u_k\\
&=\frac{\tilde{\lambda}_{1,k}}{\tilde{\lambda}_1}+\phi'V_k+\psi' u_k,
\end{aligned}
\end{equation}
where $V_k:=u_p u_{\bar{p}k}+u_{pk}u_{\bar{p}}+(\alpha^{p\bar{q}})_k u_p u_{\bar{q}}$.

Next, we will only make a small change to the proof of Theorem 2.2 in \cite{STW17}.
One difference is that when we estimate the term $F^{kk}g_{k\bar{k}1\bar{1}}$, we need to apply $\nabla_j$ and $\nabla_{\bar{j}}$ to the equation \eqref{en4.3} and set $j=1$, since $h\in \{h\in C^2(M)|\;|h|_2<\infty\}$ and assuming $\lambda_1\gg \sup_M|\nabla u|_\alpha^2+1$, we have
\begin{equation}\label{e47.20}
\begin{aligned}
F^{pq,rs}\nabla_1 g_{q\bar{p}}\nabla_{\bar{1}}g_{s\bar{r}}+F^{kk}\nabla_{\bar{1}}\nabla_1g_{k\bar{k}}&=h_{1\bar{1}}+\epsilon u_{1\bar{1}}\\
&\geq -C\mathcal{F}\lambda_1.
\end{aligned}
\end{equation}
By using \eqref{e47.20}, we can also get the inequality (3.18) in \cite{STW17}.

If $i$ is large enough(with respect to $x_i$), by (3.28) of \cite{STW17} and \eqref{eb4.39}, we have
\begin{equation}\label{eb4.41}
\begin{aligned}
0\geq -&\frac{F^{pq,rs}\nabla_1 g_{q\bar{p}}\nabla_{\bar{1}}g_{s\bar{r}}}{\lambda_1}-\frac{F^{kk}|\tilde{\lambda}_{1,k}|^2}{\lambda_1^2}\\
&+\sum_p \frac{F^{kk}}{6K}(|u_{pk}|^2+|u_{p\bar{k}}|^2)+\phi''F^{kk}|V_k|^2\\
&+\psi''F^{kk}|u_k|^2+\psi'F^{kk}u_{k\bar{k}}-C(F^{kk}\lambda_1^{-1}|g_{1\bar{1}k}|+\mathcal{F}).
\end{aligned}
\end{equation}
where $\nabla$ is the Chern connection of $\alpha$, and $\mathcal{F}=\sum_k F^{kk}$.

We now deal with two cases separately, depending on a small constant $\delta=\delta_{D_1,D_2}>0$ to be determined shortly, and which will depend on the constants $D_1$ and $D_2$.

\emph{Case} 1. Assume $\delta \lambda_1\geq -\lambda_n$. Define the set
\begin{equation}
I=\{i:F^{ii}>\delta^{-1}F^{11}\}.
\end{equation}
From \eqref{eb4.40} and Cauchy-Schwarz inequality, we get
\begin{equation}\label{eb4.43}
\begin{aligned}
-\sum_{k\notin I}\frac{F^{kk}|\tilde{\lambda}_{1,k}|^2}{\lambda_1^2}&=-\sum_{k\notin I}F^{kk}|\phi'V_k+\psi'u_k-\widetilde{H}_k|^2\\
&=-(\phi')^2\sum_{k\notin I}F^{kk}|V_k|^2-(\psi')^2\sum_{k\notin I}F^{kk}|u_k|^2-\sum_{k\notin I}F^{kk}|\widetilde{H}_k|^2\\
&\;\;\;\;-2\sum_{k\notin I}F^{kk}{\rm Re}(\phi'V_k\overline{\psi'u_k})+2\sum_{k\notin I}F^{kk}{\rm Re}(\widetilde{H}_k\overline{\psi'u_k})\\
&\;\;\;\;+2\sum_{k\notin I}F^{kk}{\rm Re}(\widetilde{H}_k\overline{\phi'V_k)}\\
&\geq-(\phi')^2\sum_{k\notin I}F^{kk}|V_k|^2-(\psi')^2\sum_{k\notin I}F^{kk}|u_k|^2-\sum_{k\notin I}F^{kk}|\widetilde{H}_k|^2\\
&\;\;\;\;-\frac{2}{3}(\phi')^2\sum_{k\notin I}F^{kk}|V_k|^2-\frac{3}{2}(\psi')^2\sum_{k\notin I}F^{kk}|u_k|^2\\
&\;\;\;\;-\frac{1}{3}(\phi')^2\sum_{k\notin I}F^{kk}|V_k|^2
-\frac{1}{2}(\psi')^2\sum_{k\notin I}F^{kk}|u_k|^2-C\sum_{k\notin I}F^{kk}|\widetilde{H}_k|^2\\
&\geq -\phi''\sum_{k\notin I}F^{kk}|V_k|^2-3(\psi')^2\delta^{-1}F^{11}K-C\sum_{k\notin I}F^{kk}|\widetilde{H}_k|^2.
\end{aligned}
\end{equation}

For $k\in I$ we have, in the same way,
\begin{equation}\label{eb4.44}
-2\delta\sum_{k\in I}\frac{F^{kk}|\tilde{\lambda}_{1,k}|^2}{\lambda_1^2}\geq -2\delta\phi''\sum_{k\in I}F^{kk}|V_k|^2-6\delta(\psi')^2\sum_{k\in I}F^{kk}|u_k|^2-C\sum_{k\in I}F^{kk}|\widetilde{H}_k|^2.
\end{equation}
We wish to use some of the good $\psi''F^{kk}|u_k|^2$ term in \eqref{eb4.41} to control the last term in \eqref{eb4.44}. For this, we assume that $\delta$ is chosen so small (depending on $\psi$, i.e. on $D_1, D_2$ and the maximum of $|u|$), such that
\begin{equation}\label{eb4.45}
6\delta(\psi')^2<\frac{1}{2}\psi''.
\end{equation}
Since $\psi''$ is strictly positive, such a $\delta>0$ exists.

Using this together with \eqref{eb4.43} and \eqref{eb4.44} in \eqref{eb4.41}, and we can choose some $i$ such that $|\widetilde{H}_k(x_i)|$ is sufficiently small by \eqref{eb4.38}, then we have
\begin{equation}
\begin{aligned}
0\geq -&\frac{F^{pq,rs}\nabla_1 g_{q\bar{p}}\nabla_{\bar{1}}g_{s\bar{r}}}{\lambda_1}-(1-2\delta)\sum_{k\in I}\frac{F^{kk}|\tilde{\lambda}_{1,k}|^2}{\lambda_1^2}\\
&+\sum_p \frac{F^{kk}}{6K}(|u_{pk}|^2+|u_{p\bar{k}}|^2)+\frac{1}{2}\psi''F^{kk}|u_k|^2+\psi'F^{kk}u_{k\bar{k}}\\
&-3(\psi')^2\delta^{-1}F^{11}K-C(F^{kk}\lambda_1^{-1}|g_{1\bar{1}k}|+\mathcal{F}).
\end{aligned}
\end{equation}

In the sequel, we can only change the term $-2(\psi')^2\delta^{-1}F^{11}K$ in the proof of case 1 in \cite{STW17} to $-3(\psi')^2\delta^{-1}F^{11}K$. By (3.55) of \cite{STW17}, we have
\begin{equation}\label{eb4.47}
\begin{aligned}
0\geq &F^{11}\biggl(\frac{\lambda_1^2}{40K}-3(\psi')^2\delta^{-1}K\biggl)+\biggl(\frac{1}{2}\psi''+C_\varepsilon\psi'\biggl)F^{kk}|u_k|^2\\
&-C_0\mathcal{F}+\varepsilon C_0\psi'\mathcal{F}-\psi'F^{kk}(\chi_{k\bar{k}}-g_{k\bar{k}}).
\end{aligned}
\end{equation}
for a uniform $C_0$. Under the assumption that the function $\underline{u}=0$ in a $\mathcal{C}$-subsolution, and that $\lambda_1\gg 1$, we may apply Proposition 2.3 of \cite{STW17} and see that there is a uniform positive number $\kappa>0$ such that one of two possibilities occurs:

(a) We have $F^{kk}(\chi_{k\bar{k}}-g_{k\bar{k}})>\kappa \mathcal{F}$. In this case we have
\begin{equation}
0\geq F^{11}\biggl(\frac{\lambda_1^2}{40K}-3(\psi')^2\delta^{-1}K\biggl)+\biggl(\frac{1}{2}\psi''+C_\varepsilon\psi'\biggl)F^{kk}|u_k|^2-C_0\mathcal{F}+\varepsilon C_0\psi'\mathcal{F}-\psi'\kappa \mathcal{F}.
\end{equation}
We first choose $\varepsilon>0$ such that $\varepsilon C_0<\frac{1}{2}\kappa$. We then choose the parameter $D_2$ in the definition of $\psi(t)=D_1e^{-D_2 t}$ to be large enough so that
\begin{equation}
\frac{1}{2}\psi''>C_\varepsilon |\psi'|.
\end{equation}
At this point, we have
\begin{equation}
0\geq F^{11}\biggl(\frac{\lambda_1^2}{40K}-3(\psi')^2\delta^{-1}K\biggl)-C_0\mathcal{F}-\frac{1}{2}\psi'\kappa \mathcal{F}.
\end{equation}
We now choose $D_1$ so large that $-\frac{1}{2}\psi'\kappa>C_0$, which implies
\begin{equation}
\frac{\lambda_1^2}{40K}\leq 3(\psi')^2\delta^{-1}K.
\end{equation}
Note that $\delta$ is determined by the choices of $D_1$ and $D_2$, according to \eqref{eb4.45}, so we obtain the required upper bound for $\lambda_1/K$.

(b) We have $F^{11}>\kappa \mathcal{F}$. With the choices of constants made above, \eqref{eb4.47} implies that
\begin{equation}
0\geq \kappa \mathcal{F}\biggl(\frac{\lambda_1^2}{40K}-3(\psi')^2\delta^{-1}K\biggl)
-C_0\mathcal{F}+\varepsilon C_0\psi'\mathcal{F}+C_1\psi'\mathcal{F}+\psi'F^{kk}g_{k\bar{k}}.
\end{equation}
for another uniform constant $C_1$. Since $F^{kk}g_{k\bar{k}}\leq \mathcal{F}\lambda_1$, we can divide through by $\mathcal{F}K$ and obtain
\begin{equation}
0\geq \frac{\kappa \lambda_1^2}{40K^2}-C_2(1+K^{-1}+\lambda_1K^{-1}),
\end{equation}
for a uniform $C_2$. The required upper bound for $\lambda_1/K$ follows from this.

\emph{Case} 2. We now assume that $\delta\lambda_1<-\lambda_n$, with all the constants $D_1$, $D_2$ and $\delta$ fixed as in the previous case. By (3.56) of \cite{STW17}, we have
\begin{equation}\label{eb4.54}
0\geq -\frac{3}{2}\frac{F^{kk}|\tilde{\lambda}_{1,k}|^2}{\lambda_1^2}+\frac{\delta^2}{10nK}\mathcal{F}\lambda_1^2+F^{kk}\phi''|V_k|^2-C\mathcal{F}\lambda_1.
\end{equation}
Similarly to \eqref{eb4.43}, we obtain
\begin{equation}
\frac{3}{2}\frac{F^{kk}|\tilde{\lambda}_{1,k}|^2}{\lambda_1^2}\leq F^{kk}\phi''|V_k|^2+C\mathcal{F}K.
\end{equation}
Returning to \eqref{eb4.54}, we obtain, since we may assume $\lambda_1\geq K$,
\begin{equation}
0\geq \frac{\delta^2\lambda_1^2}{10nK}\mathcal{F}-C\lambda_1\mathcal{F}.
\end{equation}
Dividing by $\lambda_1\mathcal{F}$ gives the required bound for $\lambda_1/K$.

Then, we immediately deduce the bound \eqref{eb4.17}, namely
\begin{equation}
\sup_M|\sqrt{-1}\partial\bar{\partial}u|_\alpha\leq C\biggl(\sup_M|\nabla u|_\alpha^2+1\biggl).
\end{equation}
A blow-up argument (Chapter 6 in \cite{S18}) combined with a Liouville theorem (Chapter 5 in \cite{S18}), shows that $\sup_M |\nabla u|_\alpha^2\leq C$, and so we get a uniform bound $|\Delta u|\leq C$.

By rewriting equation \eqref{en4.3} in real coordinates, we can then apply the Evans-Krylov Theorem \cite{r8}(see also Theorem 17.14 of \cite{GT98}) and deduce a uniform bound
\begin{equation}
|u|_{2+\beta}\leq C\;\;\;\;{\rm on}\;\;(M,\alpha)
\end{equation}
for a uniform $0<\beta<1$, where the norm is defined in \eqref{eb2.18} with respect to $(M,\alpha)$. Differentiating the equation and applying a standard bootstrapping argument, we finally obtain uniform higher-order estimates. Then we can prove the following Theorem.
\begin{theorem}\label{t4.4}
Let $(M,\alpha)$ be a complete Hermitian manifold with bounded geometry (see Definition \ref{db2.1}) of order $k-1$, $k\geq 4$. Fix a real $(1,1)$-form $\chi$ with bounded geometry of order $k-1$ on $(M,\alpha)$, given $h\in \tilde{C}^{k-2+\beta}(M)$,
let $u\in \tilde{C}^{k+\beta}(M)$ satisfy that
\begin{equation}\label{e1.9}
F(\alpha^{i\bar{p}}(\chi_{j\bar{p}}+u_{j\bar{p}}))=h+\epsilon u,
\end{equation}
where $\beta\in (0,1)$ and $\epsilon\in[0,1]$. Suppose the $\underline{u}\in \tilde{C}^{\infty}(M)$ is a $\mathcal{C}$-subsolution for the equation \eqref{e1.9}. Suppose that the equation satisfies conditions \emph{\textbf{(a1)}}-\emph{\textbf{(a4)}}. Then, we have an estimate $|u|_{k+\beta}\leq C_k$, with constant $C_k$ depending on $k$, on the background date $M,\alpha,\chi,F,h$, on $\sup_M |u|$, and the subsolution $\underline{u}$.
\end{theorem}

If $\epsilon>0$, we can obtain the following $C^0$ estimates for \eqref{e1.1}.
\begin{theorem}\label{t4.1}
Let $(M,\alpha)$ be a complete Hermitian manifold with bounded geometry of order $l$, $l\geq 0$. Fix a real $(1,1)$-form $\chi$ on $(M,\alpha)$ such that $\lambda[\alpha^{i\bar{p}}\chi_{j\bar{p}}]\in \Gamma$, given $h\in C(M)$,
let $u\in C^2(M)$ be a bounded function satisfying
\begin{equation}
F(\alpha^{i\bar{p}}(\chi_{j\bar{p}}+u_{j\bar{p}}))=h(x)+\epsilon u,
\end{equation}
where $\epsilon$ is a positive constant. Suppose that the above equation satisfies conditions \emph{\textbf{(a1)}}-\emph{\textbf{(a4)}},
then
\begin{equation}
\sup_M |u|\leq \sup_M\frac{|F(\alpha^{i\bar{p}}\chi_{j\bar{p}})-h(x)|}{\epsilon}.
\end{equation}
\end{theorem}
\begin{proof}
First, we want to estimate $\sup_M u$, we may assume that $\sup_M u>0$. By generalized maximum principle (see Proposition \ref{pb4.2}), there exists a sequence of points $\{x_k\}$ in $M$ such that $\lim u(x_k)=\sup u$ and $\overline{\lim}(u_{p\bar{q}}(x_k))\leq 0$. Then, at each $x_k$, we obtain
\begin{equation}\label{e6.36}
F(\alpha^{i\bar{p}}(\chi_{j\bar{p}}+u_{j\bar{p}}))-h(x)=\epsilon u.
\end{equation}
Let $k\rightarrow \infty$, we see immediately that
\begin{equation}
\sup_M |u|\leq \sup_M\frac{|F(\alpha^{i\bar{p}}\chi_{j\bar{p}})-h(x)|}{\epsilon}.
\end{equation}
The estimate on $\inf_M u$ is carried out similarly and then we can prove the Theorem.
\end{proof}

\section{Complete K\"{a}hler-Einstein metrics}

In this section, we establish the solvability of \eqref{e1.1} on certain manifolds that do not need to have bounded geometry, the key idea is to exhaust a general manifold using a family of manifolds with bounded geometry (see Lemma \ref{l7.3}). As a geometric application, we can construct complete K\"{a}hler-Einstein metrics, this requires using Wan-Xiong Shi's result on the short-time existence of the K\"{a}hler-Ricci flow to construct a complete metric with certain bounded properties (see Lemma \ref{l7.9}).

Next, we set up the continuity method to solve the equation \eqref{e1.1}, we assume that
\begin{equation}
F(\alpha^{i\bar{p}}\chi_{j\bar{p}})=0\;\;\;{\rm and}\;\;\;\lambda[\alpha^{i\bar{p}}\chi_{j\bar{p}}]\subset\subset \Gamma= \Gamma_n.
\end{equation}

Define
\begin{equation}
\Psi:\tilde{C}^{k+\alpha}(M)\times \mathbb{R}\rightarrow \tilde{C}^{k-2+\alpha}(M)
\end{equation}
by
\begin{equation}
\Psi(u,t)=F(\alpha^{i\bar{p}}(\chi_{j\bar{p}}+u_{j\bar{p}}))-th(x)-\epsilon u.
\end{equation}
The fact that $\Psi$ maps $\tilde{C}^{k+\alpha}(M)\times \mathbb{R}$ into $\tilde{C}^{k-2+\alpha}(M)$ can be easily verified using coordinate charts.

We are actually interested in the following open subset $U$ of $\tilde{C}^{k+\alpha}(M)$:
\begin{equation}
\begin{aligned}
U=\{u\in \tilde{C}^{k+\alpha}(M):(1/c)(\alpha_{j\bar{p}})<(\chi_{j\bar{p}}+u_{j\bar{p}})<c (\alpha_{j\bar{p}}), \;
&{\rm for \;some\; positive \; constant} \;c\}.
\end{aligned}
\end{equation}
To solve \eqref{e1.1}, we set $\mathcal{L}$ be the set $\{t\in [0,1]: {\rm \;there \;exists} \;u\in U \;{\rm such \;that} \; \Psi(u,t)=0\}$. It suffices to show that $1\in \mathcal{L}$. Indeed, we shall show that $\mathcal{L}=[0,1]$. First of all, observe that $\Psi(0,0)=0$, we have $0\in \mathcal{L}$. Then we shall show that $\mathcal{L}$ is both open and closed.

\emph{Openness}. To show that $\mathcal{L}$ is an open set, we shall use the usual implicit function theorem for Banach spaces. This amounts to showing that the Fr\'{e}chet derivative of $\Psi$ with respect to $u$, denoted as $\Psi_u$, has a bounded inverse.
The Fr\'{e}chet derivative $\Psi_u$ at the point $(u,t)$ is then an operator from $\tilde{C}^{k+\alpha}(M)$ into $\tilde{C}^{k-2+\alpha}(M)$ defined for any $g\in \tilde{C}^{k+\alpha}(M)$,
\begin{equation}
(\Psi_u(u,t))(g)=F^{ij}(\alpha^{a\bar{p}}(\chi_{b\bar{p}}+u_{b\bar{p}}))\alpha^{i\bar{l}}g_{j\bar{l}}-\epsilon g.
\end{equation}
Therefore, we have to show that, for any $v\in \tilde{C}^{k-2+\alpha}(M)$,
\begin{equation}\label{e7.26}
F^{ij}(\alpha^{a\bar{p}}(\chi_{b\bar{p}}+u_{b\bar{p}}))\alpha^{i\bar{l}}g_{j\bar{l}}-\epsilon g=v
\end{equation}
can be solved for $g\in \tilde{C}^{k+\alpha}(M)$ and that $|g|_{k+\alpha}\leq c|v|_{k-2+\alpha}$ for some constant $c$ independent of $v$. By the $C^2$ estimates and condition \textbf{(a3)}, we have
\begin{equation}
\tilde{c}^{-1}\delta_{ij}<(F^{ij}(\alpha^{a\bar{p}}(\chi_{b\bar{p}}+u_{b\bar{p}})))<\tilde{c}\delta_{ij}
\end{equation}
for some constant $\tilde{c}>0$.

We first show that there is at most on function $g$ in $\tilde{C}^{k+\alpha}(M)$ that satisfies equation \eqref{e7.26}. It is enough to check that $(\Psi_u(u,t))(g)=0$ and $g\in \tilde{C}^{k+\alpha}(M)$ imply $g\equiv 0$. Assume $g\in \tilde{C}^{k+\alpha}(M)$, $g$ is in particular bounded. We want to show that $g\leq 0$. The generalized maximum principle implies the existence of a sequence of points $\{x_i\}$ in $M$ such that $\lim g(x_i)=\sup g$ and $\overline{\lim}(g_{p\bar{q}}(x_i))\leq 0$. Applying this to the equation $(\Psi_u(u,t))(g)=0$, we obtain $\sup g\leq 0$. Similarly, $\inf g\geq 0$, which implies that $g\equiv 0$.

Next, we prove the existence of $g$.

Let $\Omega_i$ be an exhaustion of $M$ by compact subdomains. Assume $v\in \tilde{C}^{k-2+\alpha}(M)$, and let $g^i$ be the unique solution to
\begin{equation}
\begin{aligned}
(\Psi_u(u,t))(g^i)=v \;\;\;{\rm on}\;\Omega_i,\\
g^i=0\;\;\;{\rm on}\;\partial\Omega_i.
\end{aligned}
\end{equation}
By applying the maximum principle to $\Omega_i$, we obtain that
\begin{equation}
\sup_{\Omega_i}|g^i|\leq \frac{\sup |v|}{\epsilon}.
\end{equation}
Interior Schauder estimates applied in coordinate charts as in Definition \ref{db2.1} immediately show that a subsequence of $g^i$ will converge to some $g\in \tilde{C}^{k+\alpha}(M)$ which solves \eqref{e7.26} and that the estimate $|g|_{k+\alpha}\leq c|v|_{k-2+\alpha}$ is valid for some constant independent of $v$. This completes the proof of the openness of $\mathcal{L}$.

\emph{Closedness}. By Theorem \ref{t4.4} and Theorem \ref{t4.1}, we obtain the closedness of $\mathcal{L}$.

Next, we prove the uniqueness of the solution $u\in \tilde{C}^{k+\beta}(M)$. Suppose that $u$ and $\tilde{u}$ are two solutions of \eqref{e7.33}, it is obvious that
\[F[\alpha^{i\bar{p}}(\chi_{j\bar{p}}+u_{j\bar{p}})+\alpha^{i\bar{p}}(\tilde{u}-u)_{j\bar{p}}]-F[\alpha^{i\bar{p}}(\chi_{j\bar{p}}+u_{j\bar{p}})]=\epsilon(\tilde{u}-u)\]
Applying the generalized maximum principle to $\tilde{u}-u$, we can prove that $\tilde{u}\equiv u$ on $M$(the proof is similar to that of Theorem \ref{t4.1}).

Then, by the above argument, we can prove the following Theorem.
\begin{theorem}\label{t7.5}
Let $(M,\alpha)$ be a complete Hermitian manifold with bounded geometry of order $k-1$, where $k\geq 4$.
Fix a real $(1,1)$-form $\chi$ with bounded geometry of order $k-1$, and $\chi$ is uniformly equivalent to $\alpha$ on $(M,\alpha)$, given $h\in \tilde{C}^{k-2+\beta}(M)$, where $\beta\in (0,1)$. Consider the following equation on $M$:
\begin{equation}\label{e7.33}
F(\alpha^{i\bar{p}}(\chi_{j\bar{p}}+u_{j\bar{p}}))=h(x)+\epsilon u,
\end{equation}
where $\epsilon$ is a positive constant. Suppose that the above equation satisfies conditions \emph{\textbf{(a1)}}-\emph{\textbf{(a4)}}, $\Gamma= \Gamma_n$ and $F(\alpha^{i\bar{p}}\chi_{j\bar{p}})=0$, suppose $0$ is a $\mathcal{C}$-subsolution of \eqref{e7.33}. Then the equation \eqref{e7.33} exists a unique solution $u\in \tilde{C}^{k+\beta}(M)$ such that $\chi+\sqrt{-1}\partial\bar{\partial}u$ is uniformly equivalent to $\alpha$.
\end{theorem}

Next, we aim to weaken the conditions in Theorem \ref{t7.5} by replacing the bounded geometry condition with bounded curvatures and torsions. For the subsequent calculations, we need to construct certain holomorphic coordinate systems on $M$:
\begin{lemma}\label{l7.4}
Let $(M,g)$ be a Hermitian manifold of dimension $n$, assume that
\begin{equation}
|T|_g+|\nabla_g T|_g+|Rm|_g<\infty,
\end{equation}
then for any point $p\in M$, there exists a local holomorphic coordinate system $\{w^1,\cdots,w^n\}$ around $p$ such that $g_{i\bar{j}}(p)=\delta_{ij}$ and all derivatives of $g(p)$ up to order 2 are bounded, where the bounds depend only on $n$, $|T|_g$, $|\nabla_g T|_g$ and $|Rm|_g$. Moreover, if $T=0$(i.e. $g$ is a K\"{a}hler metric), then we have
\begin{equation}
\frac{\partial g_{i\bar{j}}}{\partial z^k }=0;\;\;\;\frac{\partial^2 g_{i\bar{j}}}{\partial z^k \partial z^l}=0
\end{equation}
for any $i,j,k,l$.
\end{lemma}
\begin{proof}
The result is local, denote the components of $g$ in the holomorphic coordinates $\{z^1,\cdots,z^n\}$ around $p$ by $(g_z)_{i\bar{j}}$. By a complex linear transformation, we may assume that $(g_z)_{i\bar{j}}(p)=\delta_{ij}$, define a new holomorphic coordinate system $\{y^1,\cdots,y^n\}$ by
\begin{equation}
z^i=y^i+A_{mn}^i y^m y^n,
\end{equation}
here $A_{mn}^i$ are constants to be chosen later, and repeated indices indicate a summation. Then using the new coordinates $\{y^1,\cdots,y^n\}$, we have
\begin{equation}
\begin{aligned}
(g_y)_{\alpha\bar{\beta}}&=\frac{\partial z_i}{\partial y^\alpha}\overline{\frac{\partial z_j}{\partial y^\beta}}(g_z)_{i\bar{j}}\\
&=(g_z)_{\alpha\bar{\beta}}+(A_{\alpha n}^i+A^i_{n \alpha})(g_z)_{i\bar{\beta}}y^n+(\overline{A_{\beta n}^j}+\overline{A^j_{n \beta}})(g_z)_{\alpha\bar{j}}\bar{y}^n+O(|y|^2).
\end{aligned}
\end{equation}
Now we choose $A_{\alpha l}^\beta=-\frac{1}{2}\partial_l (g_z)_{\alpha\bar{\beta}}(p)$, then $(g_y)_{\alpha\bar{\beta}}(p)=\delta_{\alpha\bar{\beta}}$ and
\begin{equation}
\partial_l (g_y)_{\alpha\bar{\beta}}(p)=\frac{1}{2}(\partial_l (g_z)_{\alpha\bar{\beta}}-\partial_\alpha (g_z)_{l\bar{\beta}})(p)=O(n,|T|_g),
\end{equation}
here we use the fact that
\begin{equation}
\frac{\partial}{\partial z^l}\Big|_p=\frac{\partial}{\partial y^l}\Big|_p.
\end{equation}

Next, we define another holomorphic coordinate system $\{w^1,\cdots,w^n\}$ by
\begin{equation}
y^i=w^i+B_{pqr}^iw^p w^q w^r,
\end{equation}
here $B_{pqr}^i$ are constants to be chosen later, it's obvious that
\begin{equation}\label{e7.19}
\begin{aligned}
&(g_w)_{i\bar{j}}(p)=(g_y)_{i\bar{j}}(p)=\delta_{ij};\\
\partial_l(g_w&)_{i\bar{j}}(p)=\partial_l(g_y)_{i\bar{j}}(p)=O(n,|T|_g)
\end{aligned}
\end{equation}
for all $l,i,j$. And
\begin{equation}
\begin{aligned}
(g_w)_{\alpha\bar{\beta}}&=\frac{\partial y_i}{\partial w^\alpha}\overline{\frac{\partial y_j}{\partial w^\beta}}(g_y)_{i\bar{j}}\\
&=(g_y)_{\alpha\bar{\beta}}+(B_{\alpha q r}^i+B_{q \alpha r}^i+B_{qr\alpha}^i)(g_y)_{i\bar{\beta}}w^q w^r\\
&\;\;\;\;\;+(\overline{B_{\beta q r}^j}+\overline{B_{q \beta r}^j}+\overline{B_{qr\beta}^j})(g_y)_{\alpha\bar{j}}\bar{w}^q \bar{w}^r+O(|w|^3).
\end{aligned}
\end{equation}
Choose $B_{lm\alpha}^{\beta}=-\frac{1}{3}\partial_l\partial_m(g_y)_{\alpha\bar{\beta}}$, then we have
\begin{equation}\label{e7.21}
\partial_l\partial_m(g_w)_{\alpha\bar{\beta}}(p)=\frac{2}{3}\partial_l(\partial_m (g_y)_{\alpha\bar{\beta}}-\partial_{\alpha}(g_y)_{m\bar{\beta}})(p)=O(n,|T|_g,|\nabla_g T|_g),
\end{equation}
here we use the fact that
\begin{equation}
\frac{\partial^2}{\partial w^l \partial w^m}\Big|_p=\frac{\partial^2}{\partial y^l \partial y^m}\Big|_p.
\end{equation}
Recall that
\begin{equation}\label{e7.23}
R_{i\bar{j}k\bar{l}}=-\frac{\partial^2g_{i\bar{j}}}{\partial w^k \partial \bar{w}^l}+g^{p\bar{q}}\frac{\partial g_{i\bar{q}}}{\partial w^k}\frac{\partial g_{p\bar{j}}}{\partial \bar{w}^l},
\end{equation}
then by \eqref{e7.19}, \eqref{e7.21} and \eqref{e7.23}, we can prove this lemma in the $w$-coordinates.
\end{proof}

Since it is in general not true that $(M,g)$ has bounded geometry, we proceed as in \cite{LT20}:

Let $\kappa\in (0,1)$, $f:[0,1)\rightarrow [0,\infty)$ be the function:
\begin{equation}
\label{e5.2}
f(s)=\left\{ \begin{aligned}
&0,\;\;\;\;\;\;\;\;\;\;\;\;\;\;\;\;\;\;\;\;\;\;\;\;\;\;\;\;\;\;\;\;\;\;\;\;\;\;\;s\in[0,1-\kappa];\\
&-\log\biggl[1-\biggl(\frac{s-1+\kappa}{\kappa}\biggl)^2\biggl],s\in(1-\kappa,1).
\end{aligned} \right.
\end{equation}
Let $\psi\geq 0$ be a smooth function on $\mathbb{R}^+$ such that 
\begin{equation}
\label{e5.3}
\psi(s)=\left\{ \begin{aligned}
&0,\;\;\;s\in[0,1-\kappa+\kappa^2];\\
&1,\;\;\;s\in(1-\kappa+2\kappa^2,1).
\end{aligned} \right.
\end{equation}
and $\frac{2}{\kappa^2}\geq \psi'\geq 0$. Define
\[\mathfrak{F}:=\int_0^s\psi(\tau)f'(\tau)d\tau\]
By \cite{LT20}, we have
\begin{lemma}\label{l77.3}Suppose $0<\kappa<\frac{1}{8}$. Then the function $\mathfrak{F}\geq 0$ defined above is smooth and satisfies the following:\\
\hspace*{0.5cm}\emph{(i)} $\mathfrak{F}=0$ for $s\in[0,1-\kappa+\kappa^2]$.\\
\hspace*{0.4cm}\emph{(ii)} $\mathfrak{F}'\geq 0$ and for any $k\geq 1$, $\exp(-k\mathfrak{F})\mathfrak{F}^{(k)}$ is uniformly bounded.\\
\end{lemma}

Let us recall a well-known result about the existence of exhaustion functions on complete Riemannian manifolds, see Shi \cite{shi972}.
\begin{lemma}\label{l7.2}\emph{(Shi)}Suppose $(M,g_{ij}(x))$ is an n-dimensional complete Riemannian manifold with sectional curvature bounded by $K$, i.e. $|Sect(g)|\leq K$ for some $K\geq 0$. Then there exists a smooth function $\tilde{\rho}$ defined on $M$ such that
\[\begin{aligned}
&1+\rho(x)\leq \tilde{\rho}(x)\leq C+\rho(x)\\
&|\nabla \tilde{\rho}(x)|\leq C,\;\;\;\;\forall x\in M\\
&|\nabla^2 \tilde{\rho}(x)|\leq C,\;\;\;\;\forall x\in M
\end{aligned}\]
where $\rho$ is the distance function with some fixed point and $C$ is a constant depending on $n$ and $K$.
\end{lemma}

If we assume that $|T|_g$, $|\nabla T|_g$ and $|Rm|_g$ are bounded, then by Lemma \ref{l2.1} and Lemma \ref{l7.2},
there exists a smooth real function $\tilde{\rho}$ which is uniformly equivalent to the distance function from a fixed point. For any $\rho_0>0$, let $U_{\rho_0}$ be the component of $\{x|\tilde{\rho}(x)<\rho_0\}$ that contains the fixed point as above. Hence $U_{\rho_0}$ will exhaust $M$ as $\rho_0\rightarrow \infty$. For $\rho_0>>1$, we define $\tilde{F}(x)=\mathfrak{F}(\tilde{\rho}(x)/\rho_0)$. Let $\tilde{g}=e^{2\tilde{F}}g$. Then $(U_{\rho_0},\tilde{g})$ is a complete Hermitian metric, see \cite{LT20}, and $\tilde{g}=g$ on $\{\tilde{\rho}(x)<(1-\kappa+\kappa^2)\rho_0\}$. By Lee-Tam \cite{LT20}, we have the following lemma:
\begin{lemma}\label{l7.3}\emph{(Lee-Tam)}
$(U_{\rho_0},\tilde{g})$ has bounded geometry of infinity order.
\end{lemma}

\begin{lemma}\label{l7.5}
Suppose that
\begin{equation}
\max\Big\{|T|_g, |\nabla_g T|_g, |Rm|_g, \sum_{k=0}^2|\nabla_g^k \chi|_g, \sum_{k=0}^2|\nabla_g^k h|_g\Big\}\leq K,
\end{equation}
where $\chi$ and $h$ are as in \eqref{e1.1}. Then we have
\begin{equation}\label{e7.25}
\max\Big\{|\tilde{T}|_{\tilde{g}}, |\nabla_{\tilde{g}} \tilde{T}|_{\tilde{g}}, |Rm(\tilde{g})|_{\tilde{g}}, \sum_{k=0}^2|\nabla_{\tilde{g}}^k (e^{2\tilde{F}}\chi)|_{\tilde{g}}, \sum_{k=0}^2|\nabla_{\tilde{g}}^k h|_{\tilde{g}}\Big\}\leq K'
\end{equation}
provided $\rho_0$ is large enough, where $K'$ is a constant independent of $\rho_0$. Moreover
\begin{equation}\label{e7.26}
|\nabla_{\tilde{g}} Rm({\tilde{g}})|_{\tilde{g}}+|\nabla_{\tilde{g}}^3 (e^{2\tilde{F}}\chi)|_{\tilde{g}}+|\nabla_{\tilde{g}}^3 f|_{\tilde{g}}<\infty.
\end{equation}
\end{lemma}
\begin{proof}
By Lemma \ref{l77.3} and \ref{l7.2}, we have
\begin{equation}\label{e7.27}
\begin{aligned}
|\nabla \tilde{F}|_{\tilde{g}}&=e^{-\tilde{F}}|\nabla \tilde{F}|_g\\
&=e^{-\tilde{F}}\rho_0^{-1}\mathfrak{F}'|\nabla \tilde{\rho}|\\
&\leq C_1 \rho_0^{-1},
\end{aligned}
\end{equation}
where $C_1$ is a constant independent of $\rho_0$.

Similarly,
\begin{equation}\label{e7.28}
|\nabla^2 \tilde{F}|_{\tilde{g}}\leq C_1 \rho_0^{-1}
\end{equation}
for a possible large $C_1$. For convenience, we can calculate in the local holomorphic coordinate system as in Lemma \ref{l7.4}, by \eqref{e7.27} and \eqref{e7.28}, we can prove \eqref{e7.25}.

Since $(U_{\rho_0},\tilde{g})$ has bounded geometry of infinity order , we can use the coordinate system as in Definition \ref{db2.1} to prove that $|\nabla_{\tilde{g}} Rm({\tilde{g}})|_{\tilde{g}}<\infty$. Note that $\overline{U_{\rho_0}}$ is compact, by Lemma \ref{l77.3}, we have $|\nabla_{\tilde{g}}^3 (e^{2\tilde{F}}\chi)|_{\tilde{g}}+|\nabla_{\tilde{g}}^3 f|_{\tilde{g}}<\infty$, then we can prove \eqref{e7.26}.
\end{proof}

Finally, we can prove the existence theorem for the solutions of \eqref{e1.1}:

\begin{proof}[Proof of Theorem \ref{t7.6}]
Let $\{\rho_k\}$ be a sequence such that $\rho_k\rightarrow \infty$, on $(U_{\rho_k},e^{2\tilde{F}}\alpha)$, we consider the following equation:
\begin{equation}\label{e7.30}
F((e^{2\tilde{F}}\alpha)^{i\bar{p}}((e^{2\tilde{F}}\chi)_{j\bar{p}}+(u_k)_{j\bar{p}}))=h(x)+\epsilon u_k.
\end{equation}
It is easy to verify that the above equation satisfies the conditions in Theorem \ref{t7.5}, hence we obtain a solution $u_k\in C^5(M)$ for \eqref{e7.30} such that
\begin{equation}\label{e7.31}
C^{-1}e^{2\tilde{F}}\alpha\leq e^{2\tilde{F}}\chi+\sqrt{-1}\partial\bar{\partial}u_k\leq Ce^{2\tilde{F}}\alpha,
\end{equation}
where $C$ is a positive constant. Provided $\rho_k$ is large enough, by Lemma \ref{l7.5}, the sequence $\{u_k\}$ has uniform priori estimates up to order 2 on $(U_{\rho_k},e^{2\tilde{F}}\alpha)$, hence the constant $C$ in \eqref{e7.31} is independent of $\rho_k$. Since $\tilde{F}=0$ on $\{\tilde{\rho}(x)<(1-\kappa+\kappa^2)\rho_k\}$, for each compact submanifold $\Omega$ of $M$, applying the Evans-Krylov Theorem and Schauder interior estimates, we obtain that the sequence $\{u_k\}$ has uniform priori estimates up to order $5$. By the Arzel\`{a}-Ascoli Theorem, we can take a subsequence $\{u_{k_m}\}$ of $\{u_k\}$ such that $\{u_{k_m}\}$ converges in $C^4$ on $\Omega$. Let $\{\Omega_i\}$ be an exhaustion of $M$ by compact subdomains, continuing in this way, we obtain a global solution $u\in C^4(M)$ for \eqref{e1.1} such that
\begin{equation}
C^{-1}\alpha\leq \chi+\sqrt{-1}\partial\bar{\partial}u\leq C\alpha.
\end{equation}

By differentiating equation \eqref{e1.1} and using the interior regularity theorem for linear elliptic PDEs(see \cite{GT98}, Theorem 6.17), we can prove that $u\in C^\infty(M)$.

The uniqueness of the solution can be proved by the the generalized maximum principle(see Lemma \ref{l2.5}), which is similar to the proof of Theorem \ref{t7.5}.
\end{proof}

As a geometric application of \eqref{e1.1}, we can construct complete K\"{a}hler-Einstein metrics on certain K\"{a}hler manifolds.
\begin{lemma}\label{l7.8}
Let $(M,\alpha)$ be a complete K\"{a}hler manifold. Suppose that
\begin{equation}
\sum_{k=0}^2|\nabla_{\alpha}^k Rm|_{\alpha}<\infty\;\;{\rm and}\;\;{\rm Ric}(\alpha)\leq c\alpha
\end{equation}
for some constant $c<0$. Then $M$ admits a complete K\"{a}hler-Einstein metric.
\end{lemma}
\begin{proof}
When $M$ is compact, the theorem was proved by Aubin \cite{Aubin76} and Yau \cite{Yau78}. Therefore we only need to consider the case when $M$ is noncompact.

Define $h_0=-{\rm Ric}(\alpha)$ as a K\"{a}hler metric on $M$, consider the following equation on $M$:
\begin{equation}\label{e7.34}
\det((h_0)_{i\bar{j}}+u_{i\bar{j}})=e^{u}e^h\det((h_0)_{i\bar{j}}),
\end{equation}
where
\begin{equation}\label{e77.10}
h:=\log \frac{\det \alpha}{\det (-{\rm Ric}(\alpha))}.
\end{equation}
For convenience, we can calculate in the local holomorphic coordinate system as in Lemma \ref{l7.4}, then $h_0$ is uniformly equivalent to $\alpha$ and
\begin{equation}
|Rm(h_0)|_{h_0}+\sum_{k=0}^2|\nabla_{h_0}^k h|_{h_0}<\infty.
\end{equation}
By Theorem \ref{t7.6}, the equation \eqref{e7.34} has a solution $u\in C^\infty(M)$ such that $((h_0)_{i\bar{j}}+u_{i\bar{j}})$ is uniformly equivalent to $(\alpha_{i\bar{j}})$, and
\begin{equation}
\begin{aligned}
{\rm Ric}((h_0)_{p\bar{q}}+u_{p\bar{q}})_{i\bar{j}}&=-\frac{\partial^2}{\partial z^i\partial\bar{z}^j}[\log \det((h_0)_{p\bar{q}}+u_{p\bar{q}})]\\
&=-\frac{\partial^2}{\partial z^i\partial\bar{z}^j}[u+\log \frac{\det \alpha}{\det (-{\rm Ric}(\alpha))}+\log \det((h_0)_{p\bar{q}})]\\
&=-((h_0)_{i\bar{j}}+u_{i\bar{j}}).
\end{aligned}
\end{equation}
Finally, $((h_0)_{i\bar{j}}+u_{i\bar{j}})$ is a complete K\"{a}hler-Einstein metric on $M$.
\end{proof}

Next, we use a method similar to that of Wu-Yau \cite{WY20} to remove the condition $\sum_{k=1}^2|\nabla_{\alpha}^k Rm|_{\alpha}<\infty$ in Lemma \ref{l7.8}.
\begin{lemma}\label{l7.9}
Let $(M,\omega)$ be an $n$-dimensional complete noncompact K\"{a}hler manifold such that
\begin{equation}\label{e7.39}
|Rm|\leq \kappa_1,\;\;\;\;{\rm Ric}(\omega)\leq -\kappa_2\omega
\end{equation}
for two constants $\kappa_1,\kappa_2>0$. Then, there exists another K\"{a}hler metric $\tilde{\omega}$ such that satisfying
\begin{equation}\label{e7.40}
C^{-1}\omega\leq \tilde{\omega}\leq C\omega,
\end{equation}
\begin{equation}\label{e7.41}
{\rm Ric}(\tilde{\omega})\leq -\tilde{\kappa}_2\tilde{\omega},
\end{equation}
\begin{equation}\label{e7.42}
|\tilde{\nabla}^q \tilde{R}m|_{\tilde{\omega}}\leq C_q,\;\;\;q=0,1,2,\ldots
\end{equation}
where $\tilde{\nabla}^q \tilde{R}m$ denotes the qth covariant derivatives of the curvature tensor $\tilde{R}m$ of $\tilde{\omega}$ with respect to $\tilde{\omega}$, and the positive constants $C=C(n)$, $\tilde{\kappa}_2=\tilde{\kappa}_2(n,\kappa_1,\kappa_2)$, $C_q=C_q(n,q,\kappa_1)$ depend only on the parameters in their parentheses.
\end{lemma}
\begin{proof}
In Lemma 13 of \cite{WY20}, the authors use some results from Shi \cite{shi97} to obtain that there exists a K\"{a}hler metric $g(x,t)$ satisfying the following K\"{a}hler-Ricci flow:
\begin{equation}
\left\{ \begin{aligned}
   &\frac{\partial }{\partial t}{g_{\alpha\overline{\beta}}}(x,t)=-4R_{\alpha\overline{\beta}} (x,t), \\
      &g_{\alpha\overline{\beta}}(x,0)= g_{\alpha\overline{\beta}}(x)
\end{aligned} \right.
\end{equation}
for all $0\leq t\leq T$, and the form $\omega(x,t)=\sqrt{-1}{g_{\alpha\overline{\beta}}}(x,t)dz^\alpha\wedge d\bar{z}^\beta$ satisfies \eqref{e7.40} and \eqref{e7.42} when $0<t\leq T$. Where $T:=C_1(n)/\kappa_1$ and $g(x)$ is the K\"{a}hler metric with respect to $\omega$. Here and below, we denote some positive constants depending only on $n$ by $C_j(n)$.

Next we will show that there exists a small $0<t_0\leq T$ such that $\omega(x,t)$ also satisfies \eqref{e7.41} whenever $0<t\leq t_0$. Recall that the curvature tensor satisfies the evolution equation (see, for example, Shi \cite{shi97}, p. 143, (122))
\[
\begin{aligned}
\frac{\partial}{\partial t}R_{\alpha\bar{\beta}\gamma\bar{\sigma}}=&4\Delta R_{\alpha\bar{\beta}\gamma\bar{\sigma}}+4g^{\mu\bar{\nu}}g^{\rho\bar{\tau}}(R_{\alpha\bar{\beta}\mu\bar{\tau}}R_{\gamma\bar{\sigma}\rho\bar{\nu}}+R_{\alpha\bar{\sigma}\mu\bar{\tau}}R_{\gamma\bar{\beta}\rho\bar{\nu}}-R_{\alpha\bar{\nu}\gamma\bar{\tau}}R_{\mu\bar{\beta}\rho\bar{\sigma}})\\
&-2g^{\mu\bar{\nu}}(R_{\alpha\bar{\nu}}R_{\mu\bar{\beta}\rho\bar{\tau}}+R_{\mu\bar{\beta}}R_{\alpha\bar{\nu}\rho\bar{\tau}}+R_{\gamma\bar{\nu}}R_{\alpha\bar{\beta}\mu\bar{\sigma}}+R_{\mu\bar{\sigma}}R_{\alpha\bar{\beta}\rho\bar{\nu}}),
\end{aligned}
\]
where $\Delta=\Delta_{\omega(x,t)}=\frac{1}{2}g^{\alpha\bar{\beta}}(x,t)(\nabla_{\bar{\beta}}\nabla_\alpha+\nabla_\alpha\nabla_{\bar{\beta}})$. Then we obtain the evolution equation for Ricci curvature:
\begin{equation}
\begin{aligned}
\frac{\partial}{\partial t}R_{\alpha\bar{\beta}}&=\frac{\partial}{\partial t}(g^{\gamma\bar{\sigma}}R_{\alpha\bar{\beta}\gamma\bar{\sigma}})\\
&=4\Delta R_{\alpha\bar{\beta}}+4g^{\gamma\bar{\sigma}}g^{\mu\bar{\nu}}g^{\rho\bar{\tau}}(R_{\alpha\bar{\beta}\mu\bar{\tau}}R_{\gamma\bar{\sigma}\rho\bar{\nu}}+R_{\alpha\bar{\sigma}\mu\bar{\tau}}R_{\gamma\bar{\beta}\rho\bar{\nu}}-R_{\alpha\bar{\nu}\gamma\bar{\tau}}R_{\mu\bar{\beta}\rho\bar{\sigma}})\\
&\;\;\;\;-2g^{\gamma\bar{\sigma}}g^{\mu\bar{\nu}}(R_{\alpha\bar{\nu}}R_{\mu\bar{\beta}\rho\bar{\tau}}+R_{\mu\bar{\beta}}R_{\alpha\bar{\nu}\rho\bar{\tau}}+R_{\gamma\bar{\nu}}R_{\alpha\bar{\beta}\mu\bar{\sigma}}+R_{\mu\bar{\sigma}}R_{\alpha\bar{\beta}\rho\bar{\nu}})\\
&\;\;\;\;-4g^{\gamma\bar{\rho}}g^{\nu\bar{\sigma}}R_{\nu\bar{\rho}}R_{\alpha\bar{\beta}\gamma\bar{\sigma}}.
\end{aligned}
\end{equation}
Then for any $\eta\in T^{1,0}_xM$, we have
\begin{equation}\label{7.45}
\begin{aligned}
(\frac{\partial}{\partial t}R_{\alpha\bar{\beta}})\eta^{\alpha}\bar{\eta}^{\beta}&\leq 4(\Delta R_{\alpha\bar{\beta}})\eta^{\alpha}\bar{\eta}^{\beta}+C_2(n)|\eta|^2_{\omega(x,t)}|Rm(x,t)|_{\omega(x,t)}^2\\
&\leq 4(\Delta R_{\alpha\bar{\beta}})\eta^{\alpha}\bar{\eta}^{\beta}+C_2(n)C_0^2|\eta|^2_{\omega(x,t)},
\end{aligned}
\end{equation}
where $C_0$ is the constant in \eqref{e7.42}. Let
\begin{equation}
{\rm Ric}(x,\eta,t)=\frac{R_{\alpha\bar{\beta}}\eta^{\alpha}\bar{\eta}^{\beta}}{|\eta|_{\omega(x,t)}^2}.
\end{equation}
Then by \eqref{e7.39} and \eqref{e7.42},
\[{\rm Ric}(x,\eta,0)\leq -\kappa_2,\;\;\;\;|{\rm Ric}(x,\eta,t)|\leq C_3(n).\]
Define
\begin{equation}
r(x,t)=\max\{{\rm Ric}(x,\eta,t);|\eta|_{\omega(x,t)}=1\}
\end{equation}
for all $x\in M$ and $0\leq t\leq T$. Then, $r$ with \eqref{7.45} satisfy all conditions in Appendix B. Lemma \ref{lb1} (the maximum principle). It follows that
\begin{equation}
r(x,t)\leq C_4(n)C_0^2t-\kappa_2.
\end{equation}
Denote
\begin{equation}
t_0=\min\Big\{\frac{\kappa_2}{2C_4(n)C_0^2},T\Big\}.
\end{equation}
Then, for all $0<t\leq t_0$,
\begin{equation}
{\rm Ric}(x,\eta,t)\leq r(x,t)\leq -\frac{\kappa_2}{2}<0.
\end{equation}
Finally, for any $0<t\leq t_0$, the form $\omega(x,t)=\sqrt{-1}{g_{\alpha\overline{\beta}}}(x,t)dz^\alpha\wedge d\bar{z}^\beta$ satisfies \eqref{e7.40}, \eqref{e7.41} and \eqref{e7.42}.
\end{proof}

Finally, we can prove Theorem \ref{t1.1}.
\begin{proof}[Proof of Theorem \ref{t1.1}]
By Lemma \ref{l7.8} and Lemma \ref{l7.9}, we obtain the existence of such a K\"{a}hler-Einstein metric up to scaling. The uniqueness of a complete K\"{a}hler
Einstein metric of negative scalar curvature follows immediately from the Schwarz Lemma, see Proposition 5.5 of \cite{CY80}.
\end{proof}
The above theorem removes two conditions in Theorem 8.2 of \cite{CY80}: 1.$(M,\alpha)$ has positive injectivity radius; 2.the covariant derivatives of $Rm$ have bounded lengths.

\section{Gauduchon conjecture on complete complex surfaces}

Tian-Yau \cite{TY90,TY91} obtained some existence results for the solutions of Monge-Amp\`{e}re equations on noncompact K\"{a}hler manifolds. In this section, we generalize their results to Hermitian manifolds $(M,\omega)$ with $\partial\overline{\partial}\omega=0$, $\partial\overline{\partial}\omega^2=0$. If $\dim M=2$, our results can be regarded as noncompact versions of the Gauduchon conjecture on certain complete complex surfaces, see the Introduction.
\begin{theorem}\label{t8.1}
Let $(M^n,\omega)$ be a complete noncompact Hermitian manifold with $\partial\overline{\partial}\omega=0$, $\partial\overline{\partial}\omega^2=0$ and $|T|_{\omega}+|\nabla_{\omega} T|_{\omega}+|Rm|_{\omega}<\infty$. Let $f$ be a smooth function on $M$. Assume the following:

\emph{(a)} the following Sobolev inequality is true
\begin{equation}\label{Sobolev}
\Big(\int_M|\phi|^{\frac{2n}{n-1}}dV\Big)^{\frac{n-1}{n}}\leq C_1\int_M |\nabla\phi|^2dV
\end{equation}
for some $C_1>0$ and all $\phi\in C_0^\infty(M)$.

\emph{(b)}
\begin{equation}
|f|(x)\leq \frac{C_2}{1+\rho_0^{2+\epsilon}(x)};\;\;\;\sum_{k=1}^2|\nabla_{\omega}^k f|_{\omega}<\infty
\end{equation}
for some $C_2,\epsilon>0$, and all $x\in M$. Where $\rho_0$ is the distance function from a fixed $o\in M$.

\emph{(c)} There exists a constant $C_3>0$ such that
\begin{equation}
Vol(B_r(x))\leq C_3 r^{2n}
\end{equation}
for some $C_3>0$ and all $r$ where $Vol(B_r(x))$ is the volume of the geodesic ball with radius $r$ centered at some $x\in M$.

Then there is a solution $\varphi\in C^{\infty}(M)$ of the following complex Monge-Amp\`{e}re equation
\begin{equation}\label{e8.1}
\left\{ \begin{aligned}
(&\omega+\sqrt{-1}\partial\overline{\partial}\varphi)^n=e^f\omega^n,\\
&\omega+\sqrt{-1}\partial\overline{\partial}\varphi>0.
\end{aligned} \right.
\end{equation}
such that $\omega+\sqrt{-1}\partial\overline{\partial}\varphi$ is uniformly equivalent to $\omega$. In addition, if $Vol(B_r(x))>c(r)>0$ for any $x\in M$ and $0<r<1$, where $c(r)$ is independent of $x$, then $\varphi(x)$ converges uniformly to zero as $x$ goes to infinity and $\varphi(x)$ is a unique solution of \eqref{e8.1}.
\end{theorem}

We shall prove this theorem in the remainder of this section. First, by Theorem \ref{t7.6}, we observe that for any $\delta>0$, the following perturbed equation is always solvable.
\begin{equation}\label{e8.2}
\left\{ \begin{aligned}
(&\omega+\sqrt{-1}\partial\overline{\partial}\varphi)^n=e^{f+\delta \varphi}\omega^n,\\
&\omega+\sqrt{-1}\partial\overline{\partial}\varphi>0.
\end{aligned} \right.
\end{equation}
Denote the unique solution of \eqref{e8.2} by $\varphi_\delta$. Then, as $\delta$ tends to zero, we want to show that $\varphi_\delta$ converges to the required solution $\varphi$ of \eqref{e8.1}. The following lemma guarantees that we can do integration by parts on $M$ for the equation \eqref{e8.2}.
\begin{lemma}\label{l8.2}
For any constants $\delta>0$, $p\geq n$, we have
\begin{equation}
\int_M |\varphi_\delta|^p \omega^n<\infty.
\end{equation}
\end{lemma}
\begin{proof}
Let $\eta$ be a cut-off function defined on $\mathbb{R}$ such that $\eta(t)\equiv 1$ for $t\leq 1$; $\eta(t)\equiv 0$ for $t\geq 2$ and $|\eta'(t)|,|\eta''(t)|\leq 2$ for all $t$.

Let $\rho(x)$ be the distance function from some fixed point $x_0$ in $M$ with respect to the metric $g$, multiplying $\eta^2\Big(\frac{\rho}{j}\Big)(1+\rho)^q\varphi_\delta |\varphi_\delta|^{p-2}$ to \eqref{e8.2} and then integrating, we obtain
\begin{equation}
\begin{aligned}
&\int_M \eta^2(1+\rho)^q\varphi_\delta |\varphi_\delta|^{p-2}((\omega+\sqrt{-1}\partial\overline{\partial}\varphi_\delta)^n-\omega^n)\\
=&\int_M (e^{f+\delta \varphi_\delta}-1)\eta^2(1+\rho)^q\varphi_\delta |\varphi_\delta|^{p-2}\omega^n.
\end{aligned}
\end{equation}
We write
\begin{equation}
\tilde{\omega}=\sum_{k=1}^n \omega^{k-1}\wedge (\omega+\sqrt{-1}\partial\overline{\partial}\varphi_\delta)^{n-k}.
\end{equation}
Note that $\partial\overline{\partial}\tilde{\omega}=0$ and
\begin{equation}
(\omega+\sqrt{-1}\partial\overline{\partial}\varphi_\delta)^n-\omega^n=\sqrt{-1}\partial\overline{\partial}\varphi_\delta\wedge\tilde{\omega},
\end{equation}
we derive by integration by parts:
\begin{equation}\label{e8.7}
\begin{aligned}
&\int_M (e^{f+\delta \varphi_\delta}-1)\eta^2(1+\rho)^q\varphi_\delta |\varphi_\delta|^{p-2}\omega^n\\
=&\sqrt{-1}\int_M \eta^2(1+\rho)^q\varphi_\delta |\varphi_\delta|^{p-2}\partial\overline{\partial}\varphi_\delta\wedge\tilde{\omega}\\
=&\sqrt{-1}\int_M \frac{1}{p}\eta^2(1+\rho)^q\overline{\partial}|\varphi_\delta|^{p}\wedge \partial \tilde{\omega}-\partial(\eta^2(1+\rho)^q\varphi_\delta |\varphi_\delta|^{p-2})\wedge \overline{\partial}\varphi_\delta\wedge \tilde{\omega}\\
=&\sqrt{-1}\int_M-\frac{1}{p}|\varphi_\delta|^{p}\overline{\partial}(\eta^2(1+\rho)^q)\wedge \partial \tilde{\omega}-\frac{1}{p}\partial(\eta^2(1+\rho)^q)\wedge \overline{\partial}|\varphi_\delta|^{p}\wedge \tilde{\omega}\\
&\;\;\;\;\;\;\;\;\;\;\;\;\;-(p-1)\eta^2(1+\rho)^q|\varphi_\delta|^{p-2}\partial \varphi_\delta\wedge\overline{\partial} \varphi_\delta \wedge \tilde{\omega}\\
=&\sqrt{-1}\int_M-\frac{2}{p}|\varphi_\delta|^{p}\overline{\partial}(\eta^2(1+\rho)^q)\wedge \partial \tilde{\omega}+\frac{1}{p}|\varphi_\delta|^{p}\partial\overline{\partial}(\eta^2(1+\rho)^q)\wedge \tilde{\omega}\\
&\;\;\;\;\;\;\;\;\;\;\;\;\;-(p-1)\eta^2(1+\rho)^q|\varphi_\delta|^{p-2}\partial \varphi_\delta\wedge\overline{\partial} \varphi_\delta \wedge \tilde{\omega}.
\end{aligned}
\end{equation}
Next, we estimate each term on the right-hand side of equation \eqref{e8.7}. By Theorem \ref{t7.6}, the solution $\varphi_\delta$ and its derivatives up to order 2 are bounded, the metric $\omega+\sqrt{-1}\partial\overline{\partial}\varphi_\delta$ is uniformly equivalent to $\omega$ on $M$. The bound and equivalence may depend on $\delta$. In the following $C_{\delta,p,q}$ will denotes a positive constant depending on $\delta,p,q$. By straightforward calculations, we get
\begin{equation}\label{e8.8}
\begin{aligned}
&\sqrt{-1}\int_M-\frac{2}{p}|\varphi_\delta|^{p}\overline{\partial}(\eta^2(1+\rho)^q)\wedge\partial \tilde{\omega}\\
=&\sqrt{-1}\int_M-\frac{2}{p}|\varphi_\delta|^{p}q\eta^2(1+\rho)^{q-1}\overline{\partial}\rho\wedge\partial \tilde{\omega}+2(1+\rho)^q\frac{\eta\eta'}{j}\overline{\partial}\rho\wedge\partial \tilde{\omega}\\
\leq &C_{\delta,p,q}\int_M(1+\rho)^{q-1}|\varphi_\delta|^{p}\omega^n,
\end{aligned}
\end{equation}
here we use the fact that $\partial \tilde{\omega}$ depends only on $\omega$ and the derivatives of $\varphi_\delta$ up to order 2. Similarly, we obtain
\begin{equation}\label{e8.9}
\sqrt{-1}\int_M \frac{1}{p}|\varphi_\delta|^{p}\partial\overline{\partial}(\eta^2(1+\rho)^q)\wedge \tilde{\omega}\leq C_{\delta,p,q}\int_M(1+\rho)^{q-1}|\varphi_\delta|^{p}\omega^n,
\end{equation}
\begin{equation}\label{e8.10}
\partial \varphi_\delta\wedge\overline{\partial} \varphi_\delta \wedge \tilde{\omega}\geq C_{\delta,p,q}|\partial \varphi_\delta|_\omega^2 \omega^n\geq 0.
\end{equation}
By \eqref{e8.7}-\eqref{e8.10}, we have
\begin{equation}\label{e8.11}
\int_M (e^{f+\delta \varphi_\delta}-1)\eta^2(1+\rho)^q\varphi_\delta |\varphi_\delta|^{p-2}\omega^n\leq C_{\delta,p,q}\int_M(1+\rho)^{q-1}|\varphi_\delta|^{p}\omega^n.
\end{equation}

On the other hand, one can easily verify that
\begin{equation}
(e^{\delta \varphi_\delta}-1)\varphi_\delta\geq \frac{\delta}{2}e^{-\delta\sup_M|\varphi_\delta|}|\varphi_\delta|^2\;\;\;\;{\rm on} \;M;
\end{equation}
therefore,
\begin{equation}\label{e8.13}
\begin{aligned}
&\int_M (e^{f+\delta \varphi_\delta}-1)\eta^2(1+\rho)^q\varphi_\delta |\varphi_\delta|^{p-2}\omega^n\\
\geq &\frac{\delta}{2}e^{-\delta\sup_M|\varphi_\delta|+\inf_M f}\int_M \eta^2(1+\rho)^q|\varphi_\delta|^{p}\omega^n
-\int_M |e^f-1|\eta^2(1+\rho)^q|\varphi_\delta|^{p-1}\omega^n.
\end{aligned}
\end{equation}
Since $f=O(\rho^{-2-\varepsilon})$ and $p\geq n$, by Young's inequality, we have
\begin{equation}\label{e8.14}
\begin{aligned}
\int_M |e^f-1|\eta^2(1+\rho)^q|\varphi_\delta|^{p-1}\omega^n\leq C_{\delta,p,q}\int_M \eta^2(1+\rho)^{q-2-\varepsilon}|\varphi_\delta|^{p-1}\omega^n\\
\leq C_{\delta,p,q}\int_M \eta^2(1+\rho)^{q-\varepsilon}\Big(\frac{p-1}{p}|\varphi_\delta|^{p}+\frac{1}{p}(1+\rho)^{-2p}\Big)\omega^n\\
\leq C_{\delta,p,q}\Big(\int_M \eta^2(1+\rho)^{q-\varepsilon}|\varphi_\delta|^{p}\omega^n+\int_M \eta^2(1+\rho)^{q-\varepsilon-2p}\omega^n\Big).
\end{aligned}
\end{equation}
Note that $Vol(B_\rho(x))\leq C\rho^{2n}$, for $p\geq n, q\leq 0$, we have
\begin{equation}\label{e8.15}
\begin{aligned}
\int_M \eta^2(1+\rho)^{q-\varepsilon-2p}\omega^n&\leq C\int_0^\infty (1+\rho)^{-2n-\varepsilon}dVol(B_\rho(x))\\
&=C(1+\rho)^{-2n-\varepsilon}Vol(B_\rho(x))|_0^\infty-C\int_0^\infty Vol(B_\rho(x))d(1+\rho)^{-2n-\varepsilon}\\
&< \infty.
\end{aligned}
\end{equation}
Here and below, we denote a positive constant that is independent of $\delta,p,q$ by $C$, which may change from line to line. By \eqref{e8.11}, \eqref{e8.13}, \eqref{e8.14} and \eqref{e8.15}, we have
\begin{equation}
\int_M \eta^2(1+\rho)^q|\varphi_\delta|^{p}\omega^n\leq C_{\delta,p,q}\Big(\int_M (1+\rho)^{q-\varepsilon}|\varphi_\delta|^{p}\omega^n+1\Big),
\end{equation}
let $j\rightarrow \infty$, we get
\begin{equation}\label{e8.17}
\int_M (1+\rho)^q|\varphi_\delta|^{p}\omega^n\leq C_{\delta,p,q}\Big(\int_M (1+\rho)^{q-\varepsilon}|\varphi_\delta|^{p}\omega^n+1\Big).
\end{equation}
Since $\varphi_\delta$ is bounded, the integral $\int_M (1+\rho)^q|\varphi_\delta|^{p}\omega^n$ will be finite if $q$ is sufficiently negative(see \eqref{e8.15}). Thus by using \eqref{e8.17} inductively, our lemma can be proved.
\end{proof}

Choose a $p_0>n$ such that
\begin{equation}
\frac{p_0+1}{n+p_0}(2+\varepsilon)>2.
\end{equation}
Then by $Vol(B_\rho(x))\leq C\rho^{2n}$ and $f=O(\rho^{-2-\varepsilon})$, we obtain
\begin{equation}\label{e8.19}
\Big(\int_M |e^f-1|^{\frac{n(p_0+1)}{n+p_0}}\omega^n\Big)^{\frac{n+p_0}{n(p_0+1)}}<\infty.
\end{equation}
Rewrite \eqref{e8.2} as
\begin{equation}\label{e8.20}
-\sqrt{-1}\partial\overline{\partial}\varphi_\delta\wedge\tilde{\omega}=(1-e^{f+\delta \varphi_\delta})\omega^n.
\end{equation}
Let $\eta$ be the cut-off function as in the proof of Lemma \ref{l8.2}. Multiplying $\eta^2\Big(\frac{\rho(x)}{R}\Big)\varphi_\delta |\varphi_\delta|^{p-1}$ with $p\geq p_0$ to both sides of \eqref{e8.20} and integrating by parts(similar to the proof of Lemma \ref{l8.2}), we have
\begin{equation}\label{e8.21}
\begin{aligned}
\int_M\Big|\nabla\Big(\eta\Big(\frac{\rho}{R}\Big)\varphi_\delta^{\frac{p+1}{2}}\Big)\Big|^2\omega^n\leq n\int_M\partial\Big(\eta\Big(\frac{\rho}{R}\Big)\varphi_\delta^{\frac{p+1}{2}}\Big)\wedge\overline{\partial}\Big(\eta\Big(\frac{\rho}{R}\Big)\varphi_\delta^{\frac{p+1}{2}}\Big)\wedge\tilde{\omega}\\
\leq np\Big(\int_M \eta^2 |\varphi_\delta|^{p}|e^f-1|\omega^n+\int_M \eta^2 |\varphi_\delta|^{p-1}\varphi_\delta (1-e^{\delta \varphi_\delta})e^f \omega^n\Big)\\
+C_{p,\delta}\Big(\frac{1}{R^2}\int_M |\varphi_\delta|^{p+1} \omega^n+\frac{1}{R}\int_M |\varphi_\delta|^{p} \omega^n\Big)
\end{aligned}
\end{equation}
Here we use the fact that $\tilde{\omega}\geq \omega^{n-1}$. By Lemma \ref{l8.2}, the last term in the above inequality tends to zero as $R\rightarrow\infty$. Note that $\varphi_\delta (1-e^{\delta \varphi_\delta})\leq 0$, applying Sobolev inequality \eqref{Sobolev} to the left-hand side of \eqref{e8.21} and letting $R\rightarrow\infty$, we obtain
\begin{equation}\label{e8.22}
\Big(\int_M |\varphi_\delta|^{(p+1)\frac{n}{n-1}}\omega^n\Big)^{\frac{n}{n-1}}\leq Cp\int_M |\varphi_\delta|^{p}|e^f-1|\omega^n.
\end{equation}
Let $p=p_0$, applying H$\ddot{\rm o}$lder inequality to the right-hand side of \eqref{e8.22}, then by \eqref{e8.19} and \eqref{e8.22}, we have
\begin{equation}\label{e8.23}
\int_M |\varphi_\delta|^{(p_0+1)\frac{n}{n-1}}\omega^n\leq C.
\end{equation}
Let $p_{k+1}=(p_k+1)\frac{n}{n-1}$ for $k\in \mathbb{N}$. Noting that $f$ is bounded, and using inequality \eqref{e8.22}, we obtain
\begin{equation}\label{e8.24}
||\varphi_\delta||_{L^{p_{k+1}}(M)}+1\leq (Cp)^{\frac{1}{p}}(||\varphi_\delta||_{L^{p_{k}}(M)}+1).
\end{equation}
Letting $k$ go to infinity, and using \eqref{e8.23}, the $L^\infty$ estimates for $\varphi_\delta$ follow from an iteration of the inequality \eqref{e8.24},
\begin{equation}\label{e8.25}
\sup_M|\varphi_\delta|\leq C,
\end{equation}
where the constant $C$ in \eqref{e8.25} is independent of $\delta$.

\begin{proof}[Proof of Theorem \ref{t8.1}]
Let $\{\delta_m\}$ be a sequence such that $\delta_m\in (0,1)$ and $\delta_m\rightarrow 0$, we obtain a sequence of the following equations,
\begin{equation}\label{e8.29}
\left\{ \begin{aligned}
(&\omega+\sqrt{-1}\partial\overline{\partial}\varphi)^n=e^{f+\delta_m \varphi}\omega^n,\\
&\omega+\sqrt{-1}\partial\overline{\partial}\varphi>0.
\end{aligned} \right.
\end{equation}
By \eqref{e8.25}, Theorem \ref{t4.4} and Theorem \ref{t7.6}, for each $m$, the equation \eqref{e8.29} exists a solution $\varphi_m\in C^{\infty}(M)$ and
\[|\varphi_m|+|\partial \varphi_m|_\omega+|\partial\overline{\partial}\varphi_m|_\omega<c,\]
where $c$ is a constant independent of $m$. The rest of the proof is similar to that of Theorem \ref{t7.6}, one can approximate the solution $\varphi$ of the original equation via the sequence $\{\varphi_m\}$.

In the remainder of the proof, we assume that $Vol(B_r(x))>c(r)>0$ for any $x\in M$ and $0<r<1$. By \eqref{e8.23} we have
\[\lim_{x\rightarrow\infty}\int_{B_r(x)} |\varphi|^{(p_0+1)\frac{n}{n-1}}\omega^n=0\]
for any $r>0$, note that $|\partial \varphi|_\omega<\infty$, we conclude that $\varphi(x)$ converges uniformly to zero as $x$ goes to infinity. The uniqueness of such a $\varphi$ follows directly from maximum principle.
\end{proof}

\begin{proof}[Proof of Theorem \ref{t1.4}]
The theorem follows immediately by applying the operator $-\sqrt{-1}\partial\overline{\partial}\log$ to \eqref{e8.1}.
\end{proof}

Tian-Yau solved \eqref{t8.1} on certain K\"{a}hler manifolds that satisfy the conditions in Theorem \ref{t1.4} (see Proposition 4.1 of \cite{TY91}), we generalize their results to Hermitian manifolds $(M,\omega)$ with $\partial\overline{\partial}\omega=0$, $\partial\overline{\partial}\omega^2=0$. Next, we give several examples of manifolds that satisfy the conditions in Theorem \ref{t1.4}.
\begin{example}\label{ex5.3} (Tian-Yau \cite{TY91})
Let $M=\bar{M}\setminus D$, where $\bar{M}$ is a K\"{a}hler manifold of complex dimension $n\geq 2$ and $D$ is a neat, almost ample, admissible divisor in $\bar{M}$.
\end{example}
The manifold $M$ as in the above example admits a complete K\"{a}hler metric such that $|Rm|<\infty$, and satisfies conditions (a) and (c) in Theorem \ref{t1.4}, see \cite{TY91} for details. There are also examples where $M$ is not necessarily a K\"{a}hler manifold:
\begin{example}\label{ex5.4}
Let $(M^n,\omega)$ be a complete noncompact Hermitian manifold with $\partial\overline{\partial}\omega=0$, $\partial\overline{\partial}\omega^2=0$, ${\rm Ric}\geq 0$, $|T|_{\omega}+|\nabla_{\omega} T|_{\omega}+|Rm|_{\omega}<\infty$ and Euclidean volume growth.
\end{example}

By Theorem 1.2 of Balogh-Krist\'aly \cite{BK23}, the manifold $M$ as in the above example satisfies the Sobolev inequality \eqref{Sobolev}.

Next, we generalize Theorem 1.1 of Tian-Yau \cite{TY90} to Hermitian manifolds $(M,\omega)$ with $\partial\overline{\partial}\omega=0$,  $\partial\overline{\partial}\omega^2=0$.
\begin{theorem}\label{t8.5}
Let $(M^n,\omega)$ be a complete noncompact Hermitian manifold with $\partial\overline{\partial}\omega=0$, $\partial\overline{\partial}\omega^2=0$ and $|T|_{\omega}+|\nabla_{\omega} T|_{\omega}+|Rm|_{\omega}<\infty$. Assume that $Vol(B_r(x_0))\leq C_2 r^{2}$ for all $r>0$, and $Vol(B_1(x))\geq C_2^{-1}(1+\rho(x_0,x))^{-\beta}$ for some fixed point $x_0$ in $M$ and some positive constants $C_2,\beta$ independent of $x$, where $\rho$ is the distance function on $M$, $B_r(x_0)$ is the geodesic ball with radius $r$ centered at $x_0\in M$. Let $f$ be a smooth function on $M$ satisfying:
\begin{equation}
\int_M(e^f-1)\omega^n=0
\end{equation}
\begin{equation}
|f|(x)\leq \frac{C_2}{(1+\rho(x_0,x))^{4+2\beta}};\;\;\;\sum_{k=1}^2|\nabla_{\omega}^k f|_{\omega}<\infty,
\end{equation}
for some constant $C_2$. Then equation \eqref{e8.1} has a smooth solution $\varphi$ such that $\omega+\sqrt{-1}\partial\overline{\partial}\varphi$ is uniformly equivalent to $\omega$.
\end{theorem}
\begin{proof}
The proof is similar to that of Theorem 1.1 in \cite{TY90}. We briefly note some of the differences:

$\bullet$ In \cite{TY90}, the authors use the second order estimates and existence results for perturbed equations such as \eqref{e8.2}(see Lemma 3.6 and Lemma 3.2 of \cite{TY90}). In our case, we instead use Theorem \ref{t7.6}.

$\bullet$ In \cite{TY90}, when the authors perform $L^\infty$ estimates on the perturbation equation, integration by parts requires the use of $d\omega=0$. In our case, we instead use $\partial\overline{\partial}\omega^k=0$ for $k=1,\ldots,n$, the calculation is slightly more complicated, but one can refer to the proof of Theorem \ref{t8.1}.
\end{proof}
\begin{proof}[Proof of Theorem \ref{t1.5}]
The theorem follows immediately by applying the operator $-\sqrt{-1}\partial\overline{\partial}\log$ to \eqref{e8.1}.
\end{proof}

We give examples of manifolds that satisfy the conditions in Theorem \ref{t1.5}.
\begin{example}\label{ex5.6}
Let $M=\bar{M}\times \mathbb{C}$, where $(\bar{M},\omega)$ is a compact Hermitian manifold such that $\partial\overline{\partial}\omega=0$, $\partial\overline{\partial}\omega^2=0$. Define a K\"{a}hler form $\hat{\omega}$ on $M$ by $\hat{\omega}(u,v):=\omega(d\pi_1 u,d\pi_1 v)+\omega_e(d\pi_2 u,d\pi_2 v)$, where $\omega_e$ is the K\"{a}hler form with respect to the standard Euclidean metric on $\mathbb{C}$, $\pi_1$ and $\pi_2$ are projection maps from $M$ to $\bar{M}$ and $\mathbb{C}$, respectively. Then one can verify that $(M,\hat{\omega})$ satisfies the conditions in Theorem \ref{t8.5}. In particular, when $\dim_{\mathbb{C}} \bar{M}=1$, $\bar{M}$ is a K\"{a}hler manifold; when $\dim_{\mathbb{C}} \bar{M}=2$, by Gauduchon Theorem \cite{Gau77}, $\bar{M}$ admits a Gauduchon metric. Hence $\bar{M}$ always admits a K\"{a}hler form $\omega$ such that $\partial\overline{\partial}\omega=0$ and
$\partial\overline{\partial}\omega^2=0$ when $\dim_{\mathbb{C}} \bar{M}=1,2$.
\end{example}

\section{Hesse-Einstein metrics}

In this section, we solve a class of fully non-linear equations and construct the Hesse-Einstein metrics on certain noncompact affine manifolds. The key technology of the section is to lift the equations on an affine manifold to it's tangent bundle (which can be viewed as a Hermitian manifold), and use the a priori estimates for \eqref{e1.1} to derive the a priori estimates for some equations on affine manifolds.

Some basic results and preliminaries about affine manifolds and Hessian manifolds can be found in Appendix A. In this section, $D$ denotes the flat connection. Now we introduce the relationship between affine structures and complex structures.
Let $(M,D)$ be a flat manifold and $TM$ be the tangent bundle of $M$ with projection $\pi:TM\rightarrow M$. For an affine coordinate system $\{x^1,\ldots,x^n\}$ on $M$, we define
\begin{equation}\label{e99.2}
z^j=\xi^j+\sqrt{-1}\xi^{n+j}
\end{equation}
where $\xi^i=x^i\circ \pi$ and $\xi^{n+i}=dx^i$. Then $n$-tuples of functions given by $\{z^1,\ldots,z^n\}$ yield holomorphic coordinate systems on $TM$. For any smooth function $u$ on $M$, for any point $p\in TM$, we have
\begin{equation}
\frac{\partial (u\circ\pi)}{\partial z_i}(p)=\frac{\partial (u\circ\pi)}{\partial \bar{z_i}}(p)=\frac{1}{2}\frac{\partial u}{\partial x_i}(\pi(p)).
\end{equation}

We denote by $J_D$ the complex structure tensor of the complex manifold $TM$. For a Riemannian metric $g$ on $M$ we put
\begin{equation}
g^T=\sum_{i,j=1}^{n}(g_{ij}\circ \pi)dz^id\overline{z}^j.
\end{equation}
Then $g^T$ is a Hermitian metric on the complex manifold $(TM,J_D)$, and for any $(0,2)$-tensor $\chi$, $\chi^T$ can be defined similarly. Let $\beta (g)$ is a form associated to the Riemannian metric $g$, which in local affine coordinates is given by
\[
{{\beta }_{ij}}(g)=-{{\partial }_{i}}{{\partial }_{j}}\log \det [g(t)]=-2\kappa_{ij}.
\]
If $\beta(g)=\lambda g$ for some constant $\lambda$, then $g$ is called the Hesse-Einstein metric, which is also called Einstein K\"{a}hler affine metric by Cheng-Yau \cite{CY82}.

The following two properties describe the relationship between Hessian manifolds and K{\"a}hler manifolds.
\begin{proposition}\label{p9.1} \emph{(\cite{r11})}
Let $(M,D)$ be a flat manifold and $g$ a Riemannian metric on $M$. Then the following conditions are equivalent.

\noindent\emph{(1)} $g$ is a Hessian metric on $(M,D)$.

\noindent\emph{(2)} $g^T$ is a K{\"a}hler metric on $(TM,J_D)$.

\end{proposition}
Furthermore, we have
\begin{proposition}\label{p9.2} \emph{(\cite{r11})}
Let $(M, D, g)$ be a Hessian manifold and $R^T$ be the Riemannian curvature tensor of the K{\"a}hler manifold $(TM,J,g^T)$. Then we have
\begin{equation}
\label{sec}
R_{i\bar{j}k\bar{l}}^T=-\frac{1}{2}Q_{ijkl}\circ \pi.
\end{equation}
Let $R_{i\bar{j}}^T$ be the Ricci tensor of the K{\"a}hler manifold $(TM,J,g^T)$. Then we have
\[R_{i\bar{j}}^T=\frac{1}{4}\beta_{ij}\circ \pi.\]
\end{proposition}

Let $(M, D, g)$ be a complete affine manifold of dimension $n$. Fix a $(0,2)$-tensor $\chi$ on $(M,g)$, for any $C^2$ function $u:M\rightarrow \mathbb{R}$ we have a new $(0,2)$-tensor $\hat{g}=\chi+D du$, and we can define $A_j^i=g^{ip}\hat{g}_{jp}$. We consider the equation for $u$ as follows:
\begin{equation}\label{e99.3}
F(A)=h(x,u),
\end{equation}
for a given function $h$ on $M$, where
\begin{equation}
F(A)=f(\lambda_1,\ldots,\lambda_n)
\end{equation}
is a smooth symmetric function of the eigenvalues of $A$ and satisfies the structure conditions \textbf{(a1)}-\textbf{(a4)}. Then if $u$ is a solution of
\eqref{e99.3}, by \eqref{e99.2}, we obtain that $u\circ\pi$ satisfies the following equation:
\begin{equation}\label{e99.5}
F((g^T)^{i\bar{p}}(\chi^T_{j\bar{p}}+4(u\circ\pi)_{j\bar{p}}))=h(\pi(x),u\circ\pi)
\end{equation}
on $(TM,g^T)$. This means that we can use the a priori estimates for \eqref{e99.5} to derive the a priori estimates for \eqref{e99.3}.

Note that on affine/Hessian manifolds, we can similarly define bounded geometry, $|\cdot|_{k+\alpha}$, $\tilde{C}^{k+\alpha}(M)$, etc. analogous to Definition \ref{db2.1}. Consequently, the existence results for solutions of \eqref{e99.3} analogous to Theorems \ref{t7.5} can be obtained.

\begin{lemma}\label{l9.1}
Let $(M,g)$ be a complete affine manifold with bounded geometry of order $k-1$, where $k\geq 4$.
Fix a $(0,2)$-tensor $\chi$ with bounded geometry of order $k-1$, and $\chi$ is uniformly equivalent to $g$ on $(M,g)$, given $h\in \tilde{C}^{k-2+\beta}(M)$, where $\beta\in (0,1)$. Consider the following equation on $M$:
\begin{equation}\label{e9.7}
F(g^{ip}(\chi_{jp}+u_{jp}))=h(x)+\epsilon u,
\end{equation}
where $\epsilon$ is a positive constant. Suppose that the above equation satisfies conditions \emph{\textbf{(a1)}}-\emph{\textbf{(a4)}}, $\Gamma= \Gamma_n$ and $F(\alpha^{i\bar{p}}\chi_{j\bar{p}})=0$, suppose $0$ is a $\mathcal{C}$-subsolution of \eqref{e9.7}. Then the equation \eqref{e9.7} exists a unique solution $u\in \tilde{C}^{k+\beta}(M)$ such that $\chi+D du$ is uniformly equivalent to $g$.
\end{lemma}
\begin{proof}
First, we consider the a priori estimates for \eqref{e9.7}, if $u\in \tilde{C}^{k+\beta}(M)$ satisfies \eqref{e9.7}, then we have
\begin{equation}\label{e9.8}
F((g^T)^{i\bar{p}}(\chi^T_{j\bar{p}}+4(u\circ\pi)_{j\bar{p}}))=h(\pi(x))+u\circ\pi
\end{equation}
on $TM$. We can easily verify that \eqref{e9.8} satisfies the conditions in Theorem \ref{t7.5}, and thus obtain the a priori estimates for $u\circ\pi$ up to $C^{k+\beta}$. Consequently, we also obtain the a priori estimates for $u$ up to $C^{k+\beta}$. We may then set up the continuity method to solve equation \eqref{e9.7}; the procedure is similar to the proof of Theorem \ref{t7.5}, so we omit the proof here.
\end{proof}

The following lemma is obvious from the definition of a flat connection.
\begin{lemma}\label{l9.2}
Let $(M,D,g)$ be an affine manifold of dimension $n$, assume that $\sum_{k=0}^n|D^k g|_g<\infty$, then for any point $p\in M$, there exists a local affine coordinate system $\{w^1,\cdots,w^n\}$ around $p$ such that $g_{ij}(p)=\delta_{ij}$ and all derivatives of $g(p)$ up to order $n$ are bounded, where the bounds depend only on $n$, $\sum_{k=0}^n|D^k g|_g$.
\end{lemma}

Suppose that $\sum_{k=0}^2|D^k g|_g<\infty$, there exists a exhaustion function on $(M,D,g)$, then we can also define $(U_{\rho_0},\tilde{g})$ as in section 4, by \cite{JY25}, we have the following:
\begin{lemma}\label{l9.3}\emph{(Jiao-Yin)}
$(U_{\rho_0},\tilde{g})$ has bounded geometry of infinity order.
\end{lemma}

Similar to Lemma \ref{l7.5}, we obtain:
\begin{lemma}\label{l9.4}
Suppose that
\begin{equation}
\max\Big\{\sum_{k=0}^2|D_g^k g|_g, \sum_{k=0}^2|D_g^k \chi|_g, \sum_{k=0}^2|D_g^k h|_g\Big\}\leq K,
\end{equation}
where $\chi$ and $h$ are as in \eqref{e9.7}. Then we have
\begin{equation}
\max\Big\{\sum_{k=0}^2|D_{\tilde{g}}^k {\tilde{g}}|_{\tilde{g}}, \sum_{k=0}^2|D_{\tilde{g}}^k (e^{2\tilde{F}}\chi)|_{\tilde{g}}, \sum_{k=0}^2|D_{\tilde{g}}^k h|_{\tilde{g}}\Big\}\leq K'
\end{equation}
provided $\rho_0$ is large enough, where $K'$ is a constant independent of $\rho_0$. Moreover
\begin{equation}
|D_{\tilde{g}}^3 {\tilde{g}}|_{\tilde{g}}+|D_{\tilde{g}}^3 (e^{2\tilde{F}}\chi)|_{\tilde{g}}+|D_{\tilde{g}}^3 f|_{\tilde{g}}<\infty.
\end{equation}
\end{lemma}

\begin{lemma}\label{l9.5}
Let $(M,D,g)$ be a complete affine manifold, $\chi$ be a $(0,2)$-tensor which is uniformly equivalent to $g$, and $h$ be a smooth function.
Assume that
\begin{equation}
\sum_{k=0}^2|D_g^k g|_g+\sum_{k=0}^2|D_g^k \chi|_g+\sum_{k=0}^2|D_g^k h|_g<\infty
\end{equation}
Consider the following equation on $M$:
\begin{equation}\label{e9.12}
F(g^{ip}(\chi_{jp}+u_{jp}))=h(x)+\epsilon u,
\end{equation}
where $\epsilon$ is a positive constant. Suppose that the above equation satisfies conditions \emph{\textbf{(a1)}}-\emph{\textbf{(a4)}}, $\Gamma= \Gamma_n$ and $F(\alpha^{i\bar{p}}\chi_{j\bar{p}})=0$, suppose $0$ is a $\mathcal{C}$-subsolution of \eqref{e9.12}. Then the equation \eqref{e9.12} exists a unique solution $u\in C^{\infty}(M)$ such that $\chi+D du$ is uniformly equivalent to $g$.
\end{lemma}
\begin{proof}
Let $\{\rho_k\}$ be a sequence such that $\rho_k\rightarrow \infty$, on $(U_{\rho_k},e^{2\tilde{F}}g)$, we consider the following equation:
\begin{equation}
F((e^{2\tilde{F}}g)^{ip}((e^{2\tilde{F}}\chi)_{jp}+(u_k)_{jp}))=h(x)+\epsilon u_k.
\end{equation}
The rest of the proof is similar to that of Theorem \ref{t7.6}, both use $u_k$ to approximate the solution of the original equation, and we only need use Lemma \ref{l9.1}-\ref{l9.4} instead of Theorem \ref{t7.5}, Lemma \ref{l7.4}, \ref{l7.3}, \ref{l7.5}.
\end{proof}
Then, we will construct the Hesse-Einstein metrics on complete Hessian manifolds with negative form $\beta$.

\begin{proof}[Proof of Theorem \ref{t1.6}]
When $(M,D,g)$ is compact, the theorem was proved by Cheng-Yau \cite{CY82}. Therefore we only need to consider the case when $(M,D,g)$ is noncompact.

Since that the second Koszul form $M$ is bounded from below by a positive constant, $-\beta (g)=2\kappa(g)$ can be seen as a Hessian metric on $M$. Consider the following equation
\begin{equation}\label{e9.16}
\det(-\beta_{ij}(g)+u_{ij})=e^{u}\exp^{\{\log\frac{\det(g)}{\det(-\beta (g))}\}}\det(-\beta_{ij}(g)),
\end{equation}
we can verify that the above equation satisfies the conditions in Lemma \ref{l9.5}, then \eqref{e9.16} exists a solution $u\in C^\infty(M)$, and we have
\begin{equation}
\begin{aligned}
\beta_{ij}(-\beta(g)+D du)&=-\frac{\partial^2}{\partial x^i\partial x^j}[\log \det(-\beta_{pq}(g)+u_{pq})]\\
&=-\frac{\partial^2}{\partial x^i\partial x^j}[u+\log\frac{\det(g)}{\det(-\beta (g))}+\log \det(-\beta_{pq}(g))]\\
&=-(-\beta_{ij}(g)+u_{ij}).
\end{aligned}
\end{equation}
Then, $(-\beta(g)+D du)$ is the complete Hesse-Einstein metric as we need.

By lifting $(M,D,g)$ to its tangent bundle $(TM,J_D,g^T)$, we obtain that $g^T$ is a K\"{a}hler-Einstein metric on $TM$ with negative scalar curvature (see Proposition \ref{p9.1} and \ref{p9.2}). The completeness of the metric $g^T$ is based on Lemma 6.1 of Jiao-Yin \cite{JY25}, then the uniqueness of $g$ follows immediately from the uniqueness of $g^T$ (see Proposition 5.5 of \cite{CY80}).
\end{proof}
The above theorem removes the bounded geometry condition in Theorem 1.9 of \cite{Yin25}.

\appendix

\section{Affine/Hessian manifolds}

\begin{definition} A connection $D$ is said to be flat if the torsion tensor $T$ and the curvature tensor $R$ vanish identically. An affine manifold $(M,D)$ is a differentiable manifold endowed with a flat connection $D$.
\end{definition}
The subsequent statement for flat manifolds are well know. For the proof see \cite{r11} section 8.1.
\begin{proposition}
\label{prop1}
\

\emph{(1)}
\begin{minipage}[t]{0.92\linewidth}
Suppose that $M$ admits a flat connection $D$. Then there exist local coordinate systems on $M$ such that $D_{\partial /\partial x^i} \partial /\partial x^j=0$. The changes between such coordinate systems are affine transformations.
\end{minipage}

\emph{(2)}
\begin{minipage}[t]{0.92\linewidth}
Conversely, if $M$ admits local coordinate systems such that the changes of the local coordinate systems are affine transformations, then there exists a flat connection $D$ satisfying $D_{\partial /\partial x^i} \partial /\partial x^j=0$ for all such local coordinate systems.
\end{minipage}
\end{proposition}

Recently, Lee-Topping \cite{LT25} solved the Hamilton's pinching conjecture, they prove that a class of Riemannian manifolds is flat, see the following:

\begin{theorem}
Suppose $(M^3,g_0)$ is a complete (connected) noncompact three-dimensional Riemannian manifold with ${\rm Ric}\geq \varepsilon \mathcal{R}\geq 0$ for some $\varepsilon>0$. Then $(M^3,g_0)$ is flat.
\end{theorem}

Let $(M, D, g)$ be an affine manifold with a Riemannian metric $g$, where $\hat{\nabla}$ denotes the Levi-Civita connection of
$(M,g)$. Defining $\gamma=\hat{\nabla}-D$. Since both connections $D$ and $\hat{\nabla}$ are torsion-free, we obtain
\[\gamma_X Y=\gamma_Y X\]
Furthermore, in affine coordinate systems, the components $\gamma_{\;jk}^i$ of $\gamma$
coincide with the Christoffel symbols $\Gamma_{\;jk}^i$ associated with the Levi-Civita connection $\hat{\nabla}$.

For the pair $(g,D)$, the Hessian curvature tensor is given by $Q=D \gamma$.
A Riemannian metric $g$ on a affine manifold $(M,D)$ is called Hessian if and only if it admits local expression through $g=D d\varphi$. An affine manifold endowed with a Hessian metric is called a Hessian manifold, which is also called K\"{a}hler affine manifolds by Cheng-Yau \cite{CY82}.

Given a flat connection $D$, a local coordinate system $\{x^1,\ldots,x^n\}$ satisfying $D_{\frac{\partial}{\partial x^i}} \frac{\partial}{\partial x^j}=0$ is called an affine coordinate system with respect to $D$.
Let $(M,D,g)$ be a Hessian manifold and its metric $g$ admits local representation
\[g_{ij}=\frac{\partial^2 \phi}{\partial x^i \partial x^j}\]
where $\{x^1,\ldots,x^n\}$ is an affine coordinate system with respect to $D$. The next two results are established in \cite{r11}.

\begin{proposition}\label{p2.3}
Let $(M,D)$ be an affine manifold and $g$ a Riemannian metric on $M$. Then the following are equivalent:

\emph{(1)} $g$ is a Hessian metric

\emph{(2)} $(D_X g)(Y,Z)=(D_Y g)(X,Z)$

\emph{(3)} $\dfrac{\partial g_{ij}}{\partial x^k}=\dfrac{\partial g_{kj}}{\partial x^i}$

\emph{(4)} $g(\gamma_X Y,Z)=g(Y,\gamma_X Z)$

\emph{(5)} $\gamma_{ijk}=\gamma_{jik}$

\end{proposition}

\begin{proposition}\label{p2.4}
 Let $\hat{R}$ be the Riemannian curvature of a Hessian metric $g=D d\phi$ and $Q=D \gamma$ be the Hessian curvature tensor for $(g,D)$. Then

\emph{(1)} $Q_{ijkl}=\dfrac{1}{2}\dfrac{\partial^4 \phi}{\partial x^i \partial x^j \partial x^k \partial x^l}-\dfrac{1}{2}g^{pq}\dfrac{\partial^3 \phi}{\partial x^i \partial x^k \partial x^p}\dfrac{\partial^3 \phi}{\partial x^j \partial x^l \partial x^q}$

\emph{(2)} $\hat{R}(X,Y)=-[\gamma_X,\gamma_Y],\;\;\;\hat{R}^i_{\;jkl}=\gamma^i_{\;lm}\gamma^m_{\;jk}-\gamma^i_{\;km}\gamma^m_{\;jl}.$

\emph{(3)}
$\hat{R}_{ijkl}=\dfrac{1}{2}(Q_{ijkl}-Q_{jikl})$

\end{proposition}

\begin{definition}
\label{Kos}
Let $(M,D,g)$ be a Hessian manifold and
$v$ the volume element of $g$. The first Koszul form $\alpha$ and the second Koszul form $\kappa$ for $(D,g)$ are defined by
\[D_X v=\alpha(X)v\]
\[\kappa=D \alpha\]
\end{definition}
It follows that
\[\alpha(X)={\rm Tr} {\gamma_X}\]
and
\[\alpha_i=\frac{1}{2}\frac{\partial \log \det[g_{pq}]}{\partial x^i}=\gamma_{\;ki}^k\]
\[\kappa_{ij}=\frac{\partial \alpha_i}{\partial x^j}=\frac{1}{2}\frac{\partial \log \det[g_{pq}]}{\partial x^i \partial x^j}\]
locally.

\section{Maximum principles}

Let $(M,\omega_0)$ be an $n$-dimensional complete noncompact K\"{a}hler manifold. Suppose that $\omega(x,t)>0$ is a smooth solution for the following K\"{a}hler-Ricci flow
\begin{equation}
\left\{ \begin{aligned}
   &\frac{\partial }{\partial t}{g_{\alpha\overline{\beta}}}(x,t)=-4R_{\alpha\overline{\beta}} (x,t)\;\;\;\;{\rm on}\;M\times[0,T]\\
      &g_{\alpha\overline{\beta}}(x,0)= (g_0)_{\alpha\overline{\beta}}(x)\;\;\;\;x\in M
\end{aligned} \right.
\end{equation}
where $g(x,t)$ and $g_0$ are the K\"{a}hler metric with respect to $\omega(x,t)$ and $\omega_0$, $T$ is a positive constant. Assume that the curvature tensor $Rm(x,t)$ of $\omega(x,t)$ satisfies
\begin{equation}
\sup_{M\times[0,T]}|Rm(x,t)|^2\leq k_0
\end{equation}
for some constant $k_0>0$.
\begin{lemma}\label{lb1}
With the above assumption, let $\{W_{\alpha\bar{\beta}}\}$ be a smooth tensor on $M$ satisfying $\overline{W}_{\alpha\bar{\beta}}(x,t)=W_{\beta\bar{\alpha}}(x,t)$ and
\begin{equation}
(\frac{\partial}{\partial t}W_{\alpha\bar{\beta}})\eta^{\alpha}\bar{\eta}^{\beta}
\leq (\Delta W_{\alpha\bar{\beta}})\eta^{\alpha}\bar{\eta}^{\beta}+C_1|\eta|^2_{\omega(x,t)}
\end{equation}
for all $x\in M$, $\eta\in T^{1,0}_xM$, $0\leq t\leq T$, where $\Delta=2g^{\alpha\bar{\beta}}(x,t)(\nabla_{\bar{\beta}}\nabla_\alpha+\nabla_\alpha\nabla_{\bar{\beta}})$ and $C_1$ is a constant. Let
\begin{equation}
r(x,t)=\max\{W_{\alpha\bar{\beta}}\eta^{\alpha}\bar{\eta}^{\beta};\eta\in T^{1,0}_xM, |\eta|_{\omega(x,t)}=1\}
\end{equation}
for all $x\in M$ and $0\leq t\leq T$. Suppose
\begin{equation}
\sup_{M\times[0,T]}|r(x,t)|\leq C_0,\;\;\;\;\sup_M r(x,0)\leq -\kappa
\end{equation}
for some constants $C_0>0$ and $\kappa$. Then
\begin{equation}
r(x,t)\leq (C_0C_2(n)\sqrt{k_0}+C_1)t-\kappa
\end{equation}
for all $x\in M$, $0\leq t\leq T$, where $C_2(n)$ is a positive constant depends only on $n$.
\end{lemma}
\begin{proof}
The proof is similar to that of Lemma 15 in \cite{WY20}, we only need to consider the function
\begin{equation}
\tilde{f}(x,\eta,t)=\frac{W_{\alpha\bar{\beta}}\eta^{\alpha}\bar{\eta}^{\beta}}{|\eta|_{\omega(x,t)}^2}-Ct+\kappa
\end{equation}
instead of $f(x,\eta,t)$ in the proof of Lemma 15 in \cite{WY20}.
\end{proof}


\end{document}